\documentclass[12pt]{amsart}

\usepackage{amsmath,amsthm,amsfonts,amscd,amssymb}
\usepackage{pst-node}
\usepackage{pst-plot}
\usepackage{fullpage}

\usepackage{tikz, braids}
\usetikzlibrary{decorations.pathreplacing, arrows}

\usetikzlibrary{shapes,snakes,positioning,decorations.pathmorphing,calc,backgrounds,chains,fit,matrix,trees,scopes,shapes.geometric,shapes.multipart}
\tikzset{TL/.style={scale=.5, ineqn}}
\tikzset{JWP/.style={scale=.5, ineqn}}
\tikzset{TLnode/.style = {inner sep = 0, minimum size = 1cm}}
\tikzset{TL at/.style = {shift={($(#1)-(.5,.5)$)}, execute at begin scope = {\draw (0,0) rectangle (1,1);}}}
\tikzset{overstrand/.style={preaction={draw=white, -, line width=6pt}}}
\tikzset{normalizedvert/.style={circle,draw, minimum size = 1mm, inner sep = 0, fill=white}}
\tikzset{coordish/.style={minimum size=0, inner sep = 0, fill=black}}
\tikzset{nicelabel/.style={text height=1.5ex,text depth=.25ex}} 
\tikzset{horiznicelabel/.style={text width = .5em}} 
\tikzset{ineqn/.style={baseline = {($(current bounding box.center)-(0,1ex)$)}}}
\tikzset{ctrl/.style={controls = { #1 and #1} }}
\tikzset{fibabctree/.style={execute at begin picture={  
 \draw (-.6,0) -- node[auto,swap]{$a$} (0,0) coordinate(botvert) -- node[auto,swap]{$c$} (.6,0);
 \draw (botvert) -- node[auto,swap]{$b'$} (0,.5) coordinate(topvert);
 \foreach \sign in {+,-}
  \draw (topvert) -- +(\sign 1/6,.5) coordinate(leaf\sign);
}}}
\tikzset{fibacbase/.style={baseline=.5cm,nicelabel, execute at begin picture={  
 \draw (-.6,0) node[below right]{$a$} -- (-.25,0) coordinate(leftshadow) -- (0,0) coordinate(botvert) node[below]{$#1$} -- (.25,0) coordinate(rightshadow) -- (.6,0) node[below left]{$c$};
}}}

\tikzset{combspacing/.style = {row sep = .5cm, column sep = 1cm}}  
\tikzset{comb/.style = {combspacing, ampersand replacement = \&, row 1/.style=coordish, matrix of nodes, nodes in empty cells, nodes=draw}}

\usepackage{hyperref}

\title[Dual bases in Temperley-Lieb algebras, quantum groups, and a question of Jones]
{Dual bases in Temperley-Lieb algebras, quantum groups, and a question of Jones}

\author {Michael Brannan}
\address{Michael Brannan,
Department of Mathematics,
Mailstop 3368, Texas A\&M University, 
College Station, TX 77843-3368, USA}
\email{mbrannan@math.tamu.edu}

\author {Beno\^\i{}t Collins}
\address{Beno\^\i{}t Collins,
Department of Mathematics, Kyoto University,
and CNRS}
\email{collins@math.kyoto-u.ac.jp}

\theoremstyle{plain}
\newtheorem{lemma}{Lemma}[section]
\newtheorem{theorem}[lemma]{Theorem}
\newtheorem{proposition}[lemma]{Proposition}
\newtheorem{corollary}[lemma]{Corollary}

\newtheorem*{maintheorema}{Theorem A}
\newtheorem*{maintheoremb}{Theorem B}

\theoremstyle{definition}
\newtheorem{definition}{Definition}

\theoremstyle{remark}
\newtheorem{remark}{Remark}
\newtheorem{example}{Example}
\newtheorem{question}{Question}
\newtheorem{notation}{Notation}

\newcommand{\myop}[1]{\operatorname{#1}}

\newcommand{\Tr}{\myop{Tr}}

\newcommand{\Wg}{\myop{Wg}}

\newcommand{\N}{\mathbb N}
\newcommand{\C}{\mathbb C}

\newcommand{\R}{\mathbb R}

\newcommand{\TL}{\text{TL}}
\newcommand{\mc}{\mathcal}

\begin{document}

\begin{abstract}
We derive a Laurent series expansion for the structure coefficients appearing in the dual basis corresponding to the Kauffman diagram basis of the Temperley-Lieb algebra $\TL_k(d)$, converging  for all complex loop parameters $d$ with $|d| > 2\cos\big(\frac{\pi}{k+1}\big)$.   In particular, this yields a new formula for the structure coefficients of the Jones-Wenzl projection in $\TL_k(d)$.  The  coefficients appearing in each Laurent expansion are shown to have a natural combinatorial interpretation in terms of a certain graph structure we place on non-crossing pairings, and these coefficients turn out to have the remarkable property that they either always positive integers or always negative integers. As an application, we answer affirmatively a question of Vaughan Jones, asking whether every Temperley-Lieb diagram appears with non-zero coefficient in the expansion of each dual basis element in $\TL_k(d)$ (when $d \in \R \backslash [-2\cos\big(\frac{\pi}{k+1}\big),2\cos\big(\frac{\pi}{k+1}\big)]$).  Specializing to Jones-Wenzl projections, this result gives a new proof of a result of Ocneanu \cite{Oc02}, stating that every Temperley-Lieb diagram appears with non-zero coefficient in a Jones-Wenzl projection.  Our methods establish a connection with the Weingarten calculus on free quantum groups, and yield as a byproduct improved asymptotics for the free orthogonal Weingarten function. 
\end{abstract}

\maketitle

\section{Introduction}

The Temperley-Lieb algebras form a very important class of finite-dimensional algebras, arising in a remarkable variety of mathematical and physical contexts including lattice models \cite{TeLi71}, knot theory \cite{KaLi94}, subfactors and planar algebras \cite{JoSu97}, quantum groups \cite{Ba96, Wo87a}, and  topological quantum computation \cite{Ab08, Zh09, DeRoWa16}.  Given a complex number $d \in \C^*$ and a natural number $k \in \N$, the $k$th Temperley-Lieb algebra $\TL_k(d)$ (with loop parameter $d$) is a unital finite-dimensional complex associative algebra given by a finite set of generators $1, u_1, \ldots, u_k$ subject to the relations $u_iu_j=u_ju_i$ when $|i-j|\ge 2$, $u_iu_{i+1}u_i =  u_i$, and $u_i^2=du_i$.  These algebras admit a canonical tracial linear functional $\Tr:\TL_k(d) \to \C$, called the {\it Markov trace}.  Using the Markov trace, once can define a natural symmetric bilinear form $\langle \cdot, \cdot \rangle$ on $\TL_k(d)$, and provided $d$ is not twice the real part of a root of unity, this bilinear form is non-degenerate (see for example  \cite{KoSm91, Ca11, DiFr98}). Within this non-degenerate regime, a fundamental problem of interest is to compute explicitly the dual basis (with respect to the pairing $\langle \cdot, \cdot \rangle$) corresponding to the standard linear basis for $\TL_k(d)$ consisting of Temperley-Lieb diagrams.  See Section \ref{prelim} for precise definitions and notation.   A special case of this dual basis problem is the much studied problem of computing the coefficients appearing in the Temperley-Lieb diagram expansion of famous Jones-Wenzl projections $q_k \in \TL_k(d)$.  The Jones-Wenzl projections are certain ``highest weight''  idempotents $q_k \in \TL_k(d)$, and are key to the structure and applications of Temperley-Lieb algebras in representation theory, operator algebras and mathematical physics. Despite the importance of the Jones-Wenzl projections, remarkably very little is known about these idempotents beyond the fundamental Wenzl recursion formula \cite{We87} and its various generalizations and extensions.  See for example \cite{FrKh97, Mo15, Oc02}. 

Two fundamental questions pertaining to the Jones-Wenzl projection (and more generally any dual basis element) in $\TL_k(d)$ are: 

\begin{question} \label{q1} Is there an algorithm or formula for computing the coefficient of each Temperley-Lieb diagram appearing in such an element?
\end{question}
\begin{question}[Vaughan Jones] \label{q2}   Does each Temperley-Lieb diagram appear with non-zero coefficient in such an expansion?  
\end{question} These two questions arose in many contexts, from subfactor theory and representation theory \cite{Oc02}, to topological quantum computation \cite[Problem 3.15]{DeRoWa16}.  Over the years, some progress on these questions has been made.  Perhaps the most notable is the  announcement of a closed formula for the coefficients of the Jones-Wenzl projection $q_k$ by Ocneanu \cite{Oc02} which answered both questions in the affirmative for $q_k$, at least for real-valued loop parameters.  This formula of Ocneanu was later verified in certain special cases by Reznikoff \cite{Re02, Re07}.   Another complementary approach to the computation of the coefficients of $q_k$ was developed independently by Morrison \cite{Mo15} and Frenkel-Khovanov \cite{FrKh97}.  With regards to the more general problem of computing the coefficients of {\it arbitrary} dual basis elements in $\TL_k(d)$, essentially no prior progress seems to have been made.  In particular, it is not clear how the algebraic techniques for the coefficients of the Jones-Wenzl projections can be adapted.

In this paper, we solve these two questions for a broad class of loop parameters by connecting the problem of computing the values of the coefficients of each Temperley-Lieb diagram appearing in the expansion of a dual basis element in $\TL_k(d)$ to a seemingly different problem of computing polynomial integrals over a class of compact quantum groups, called {\it free orthogonal quantum groups}.  Using a combinatorial tool called the {\it Weingarten calculus}, we are able to interpret generic coefficients of dual basis elements (in particular, Jones-Wenzl projections) in terms of certain moments of coordinate functions over free orthogonal quantum groups taken with respect to the Haar integral.  This new operator algebraic quantum group perspective has the advantage of revealing ``hidden'' algebraic relations between the structure coefficients of the dual basis, and provides a new streamlined approach to computing the structure coefficients of {\it any} dual basis element, not just the Jones-Wenzl projection.  Using these ideas, we are able to prove the following main theorem of the paper.  See Section \ref{prelim} and Theorem \ref{thm:Laurent-series} for any undefined concepts and a more detailed restatement.  Below, $NC_2(2k)$ denotes the set of all non-crossing pair partitions of the ordered set $\{1, \ldots ,2k\}$.

\begin{maintheorema} [See Theorem \ref{thm:Laurent-series}]  \label{mt1}
Let $\{D_p\}_{p \in NC_2(2k)} \subset \TL_k(d)$ denote the linear basis of Temperley-Lieb diagrams, and denote by $\{\hat D_p\}_{p \in NC_2(2k)}$ the corresponding dual basis with respect to the bilinear form $\langle \cdot, \cdot \rangle$ induced by the Markov trace.  For each $p$, write $\hat D_p = \sum_q  \Wg_d(p,q) D_q$, where $\Wg_d(p,q) \in \C$ is the coefficient of $D_q$ in $\hat D_p$.  Then  there are positive integers $L(p,q) \in \N$ and  $\{m_r(p,q)\}_{r \in \N_0} \subset \N$ such that $\Wg_d(p,q)$ has the following Laurent series expansion
\begin{align} \label{ls}
 \Wg_d(p,q) = (-1)^{|p \vee q| +k}\sum_{r \ge 0} m_r(p,q) d^{-L(p,q)-2r} \qquad \Big(|d| > 2\cos\big(\frac{\pi}{k+1}\big) \Big).
\end{align}
\end{maintheorema}

In Section \ref{app}, we explain in detail an algorithm for computing the integers $m_r(p,q)$ and $L(p,q)$, which turn out to be combinatorially interesting objects in their own right. The numbers $m_r(p,q)$ count certain collections of paths of a given length in a directed graph built from pairs of non-crossing pair partitions, and $L(p,q)$ measures the length of a certain shortest path in this graph.  For example, when $p = \{1,4\}\{2,3\}\{5,6\}$ and $q = \{1,6\}\{2,5\}\{3,4\} \in NC_2(6)$, our algorithm turns out to produce the following directed graph:

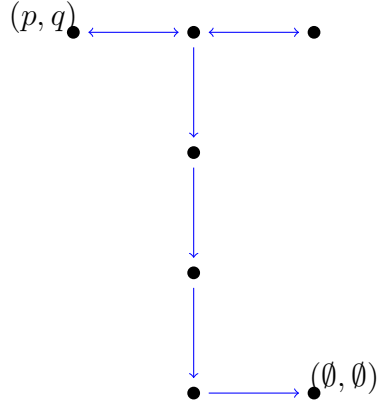
\begin{equation*}
	\begin{tikzpicture}[baseline=(current  bounding  box.center),
			wh/.style={circle,draw=black,thick,inner sep=.5mm},
			bl/.style={circle,draw=black,fill=black,thick,inner sep=.5mm}, scale = 0.4]
		
\node(a) at (-4,0) [bl] {};
\node(b) at (0,0) [bl] {};
\node(c) at (4,0) [bl] {};
\node(d) at (0,-4) [bl] {};
\node(e) at (0,-8) [bl] {};
\node(f) at (0,-12) [bl] {};
\node(g) at (4,-12) [bl] {};

\node (h) at (-5,0.5) {$(p,q)$};	
\node (i) at (5,-11.5) {$(\emptyset, \emptyset)$};			

\draw [<->, color=blue]
		(-3.5,0) -- (-.5,0);

\draw [<->, color=blue]
		(0.5,-0) -- (3.5,-0);

\draw [->, color=blue]
		(0,-4.5) -- (0,-7.5);

\draw [->, color=blue]
		(0, -8.5) -- (0,-11.5);

\draw [->, color=blue]
		(0,-0.5) -- (0,-3.5);

\draw [->, color=blue]
		(0.5,-12) -- (3.5,-12);
	\end{tikzpicture}
\end{equation*}
From this remarkably simple graph we obtain $L(p,q) = 5$ as the length of the shortest directed path from the node labeled $(p,q)$ to the node labeled $(\emptyset, \emptyset)$.  Similarly, the coefficients $m_r(p,q) = 2^{r+1}-1$ count the number of distinct directed paths of length $L(p,q) + 2r$ between these same two nodes.  Putting this data together, we obtain the formula 
\[
\Wg_d(p,q) =  \sum_{r \ge 0} (2^{r+1}-1)d^{-5 - 2r}
\] 
in this case. We refer the reader to Section \ref{app} and Example \ref{ex:workout} for the precise details of this computation.

Using the above theorem, we obtain in a uniform way a procedure for computing the dual basis $\{\hat D_p\}_{p \in NC_2(2k)}$ in the generic regime $|d|> 2\cos\big(\frac{\pi}{k+1}\big)$, providing the first significant advancement on Question \ref{q1} above.   As a byproduct of the positivity properties  of the coefficients of the Laurent series \eqref{ls}, we are also able to provide an affirmative answer to Jones' Question \ref{q2} on non-zero coefficients for the dual basis, as follows.

\begin{maintheoremb}[See Theorem \ref{thm:non-zero-coefficients}] \label{mt2}
For generic loop parameters $d$, every coefficient in the diagram expansion of the dual basis (in particular, the Jones-Wenzl projection) of $\TL_k(d)$ is non-zero.  More precisely, we have $\Wg_d(p,q) \ne 0$ when
\[
d  \in \R \backslash \Big[-2\cos\big(\frac{\pi}{k+1}\big), 2\cos\big(\frac{\pi}{k+1}\big)\Big] \quad \text{or} \quad |d| \text{ is sufficiently large}.
\]
\end{maintheoremb}

Again, specializing to the case of Jones-Wenzl projections, the above two theorems agree with and confirm the non-zero coefficients result of Ocneanu \cite{Oc02}, and complement the previous works of Morrison \cite[Proposition 5.1]{Mo15} and Frenkel-Khovanov \cite{FrKh97}.   

Finally, as mentioned above, our methods in this paper are based on exploiting a connection between the structure coefficients of the dual basis in $\TL_k(d)$ and the so called Weingarten calculus on {\it free orthogonal quantum groups}.  Very roughly, the problem of computing polynomial integrals over this class of quantum groups is encoded in a family of functions indexed by pairs of non-crossing pairings called {\it Weingarten functions}, which turn out to be exactly the coefficients $\Wg_d(p,q)$, when viewed as functions of $d \in \C$.   The study of the large $d$ asymptotics of $\Wg_d(p,q)$ (and its 
 variants for other quantum groups) has played an extremely important role in the discovery of quantum symmetries in free probability theory, and has led to non-commutative de Finetti theorems and other asymptotic freeness results for the so-called easy quantum groups.  See \cite{BaSp, BaCuSp12, CuSp11, Cu}.  With regards to the asymptotics of the Weingarten function, estimates were given \cite{BaCuSp12, CuSp11} in an attempt to isolate the order and the value of the leading term in the $\frac1d$-expansion of $\Wg_d(p,q)$.  The best  among these prior works was Theorem 4.6 in \cite{CuSp11}, which isolates the leading non-zero term in $\Wg_d(p,q)$  for certain pairs of pairings $(p,q)$.  On the other hand, it is clear that the Laurent series expansion for $\Wg_d(p,q)$ in Theorem A provides the first explicit description of the leading term for {\it all} possible pairs $(p,q)$.  In fact, we shall see in Example \ref{ex:CS-fail} how for certain pairs $(p,q)$, the leading order of $\Wg_d(p,q)$ that one might anticipate based on an examination of Theorem 4.6 in \cite{CuSp11} turns out to differ by a factor of $d^{-2}$ from the true value given by Theorem A.  In the future, the authors hope to investigate potential  applications of our refined understanding of the Weingarten functions to operator algebraic/free probabilistic aspects of free quantum groups.

\subsection{Organization of  the paper}

The rest of this paper is organized as follows.  Section \ref{prelim} gives a brief introduction to the Temperley-Lieb algebras, the Markov trace, and the Jones-Wenzl projections.  Section \ref{fqg} reviews some facts about free orthogonal quantum groups and the Weingarten calculus for computing Haar integrals over these quantum groups.  In this section we also explain how to express the coefficients of the dual Temperley-Lieb diagram basis in terms of  Weingarten functions.  The final Section \ref{app} contains our main results on the structure of the Weingarten functions and presents the applications to the dual  Temperley-Lieb basis mentioned above.  The results in this section are obtained by constructing a certain directed graph $\mc G$, with vertex set $V_\mc G = \bigsqcup_{k \in \N_0} NC_2(2k)\times NC_2(2k)$ , and edge set $E_{\mc G}$ defined so as to keep track of certain algebraic relations satisfied by the variables $\Wg_d(p,q)$ imposed by the underlying quantum group symmetries.  We call $\mc G$ the {\it Weingarten graph}, and use its structure to describe and explicitly compute the positive integer coefficients $m_r(p,q)$ and $L(p,q)$ in Theorem A.

\subsection*{Acknowledgements}  The authors thank Michael Hartglass and Vaughan Jones for fruitful conversations and comments.  
The authors also thank the University of California, Berkeley and the organizers of the 2016 {\it Free Probability and Large N Limit Workshop}, held at UCB, for a fruitful work environment where part of this work was completed.  
We are also grateful to Scott Morrison for suggestions upon reading the first version of this paper. 
B.C. was supported by NSERC discovery and accelerator grants, JSPS Kakenhi wakate B,  and ANR-14-CE25-0003.

\section{Preliminaries and notation} \label{prelim}

\subsection{Non-crossing pair partitions and the Temperley-Lieb algebras}
Given $k \in \N$, we denote by $NC_2(2k)$ the collection of non-crossing pair partitions on 
the ordered set $[2k]:=\{1,\ldots , 2k\}$. This collection has cardinal $C_k$ where
$C_k=\frac{(2k)!}{k!(k+1)!}$ is the $k$th Catalan number.  Given $p,q \in NC_2(2k)$, we write $p \vee q$ for the smallest partition of $[2k]$ such that $p, q \le p \vee q$, where $\le$ denotes the usual refinement partial order on set partitions.  We will also write $|p \vee q|$ for the number of blocks in the partition $p \vee q$.  

We now formally introduce the Temperley-Lieb algebras.  A good general reference for these objects is the book \cite{KaLi94}.   

\begin{definition}
Let $d \in \C^*$ and $k \in \N$ be fixed parameters.   The {\it Temperley-Lieb algebra} is the unital associative $\C$-algebra generated by elements $1, u_1, \ldots, u_{k-1}$ subject to the following relations 
\begin{itemize}
\item $u_iu_j=u_ju_i$ when $|i-j|\ge 2$ 
\item $u_iu_{i+1}u_i =  u_i$
\item $u_i^2=du_i$ 
\end{itemize}
\end{definition}

It is well known that the algebra $\TL_k(d)$ is always finite dimensional, with dimension equal to $C_k$ whenever $d$ is not twice the real part of a root of unity.    See for example \cite{Jo83}.  When $d \in (0,\infty)$ there is a natural conjugate-linear involution on $\TL_k(d)$, defined by declaring 
\[u_i^*=u_i \qquad (1\le  i \le k-1).\] 
Thus for $d \in (0,\infty)$, $\TL_k(d)$ may be regarded as an involutive $\ast$-algebra, and it is known that $\TL_k(d)$ admits a non-trivial $\ast$-representation into a C$^\ast$-algebra precisely when $d \in \mc D = [2, \infty) \cup \{2 \cos\big(\frac\pi n\big): n = 3,4,5, \ldots \}$ \cite{We87}.  For our purposes, however, we will not require the use of a  $\ast$-structure on $\TL_k(d)$.  


\subsection{Temperley-Lieb diagrams} With $k \in \N$ and $d \in \C^*$ fixed as above, we plot the set $[2k] = \{1,\ldots , 2k\}$ on a square clockwise with
$\{1,\ldots , k\}$ on the top edge and $\{2k,\ldots ,k+1\}$ on the bottom edge.  If we connect these points by a non-crossing pairing $p \in NC_2(2k)$, this results in a planar diagram $D_p$, called a {\it Temperley-Lieb diagram}.  For example, when $k=3$, there are $C_3 = 5$ Temperley-Lieb diagrams $\{D_p\}_{p \in NC_2(6)}$:
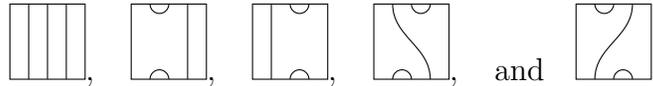
\[
  \begin{tikzpicture}[scale=1]
\draw (0,0) rectangle (1,1);
\draw (1/4,0) --(1/4,1);
\draw (1/2,0) --(1/2,1);
\draw (3/4,0) --(3/4,1);
\end{tikzpicture}, \quad
\begin{tikzpicture}[scale=1]
\draw (0,0) rectangle (1,1);
\draw (1/4,1) arc (180:360:1/8);
\draw (1/4,0) arc (180:0:1/8);
\draw (3/4,0) --(3/4,1);
\end{tikzpicture}, \quad
\begin{tikzpicture}[scale=1]
\draw (0,0) rectangle (1,1);
\draw (1/2,1) arc (180:360:1/8);
\draw (1/2,0) arc (180:0:1/8);
\draw (1/4,0) --(1/4,1);

\end{tikzpicture}, \quad
\begin{tikzpicture}[scale=1]
\draw (0,0) rectangle (1,1);
\draw (1/4,1) to [out=-90,in=90] (3/4,0);
\draw (1/2,1) arc (180:360:1/8);
\draw (1/4,0) arc (180:0:1/8);
\end{tikzpicture}, \quad \text{and} \quad
\begin{tikzpicture}[scale=1]
\draw (0,0) rectangle (1,1);
\draw (3/4,1) to [out=-90,in=90] (1/4,0);
\draw (1/4,1) arc (180:360:1/8);
\draw (1/2,0) arc (180:0:1/8);
\end{tikzpicture}.
\]

On the $\C$-vector space $\C[NC_2(2k)]$ spanned by the Temperley-Lieb diagrams $\{D_p\}_{p \in NC_2(2k)}$ we define an associative $\C$-algebra structure as follows. The product $D_pD_q$ of diagrams $D_p$ and $D_q$ is obtained by first stacking diagram $D_q$ on top of $D_p$, connecting the bottom row of $k$ points on $D_q$ to the top row of $k$ points on $D_p$.  The result is a new planar diagram, which may have a certain number $c$ of internal loops.  By removing these loops, we obtain a new diagram $D_r$ for some $r \in NC_2(2k)$ (which is unique up to planar isotopy).  The product $D_pD_q$ is then defined to be $d^c D_r$.  For example, we have
\[  \begin{tikzpicture}[scale=1]
\draw (0,0) rectangle (1,1);
\draw (1/4,1) to [out=-90,in=90] (3/4,0);
\draw (1/2,1) arc (180:360:1/8);
\draw (1/4,0) arc (180:0:1/8);
\end{tikzpicture}\times
\begin{tikzpicture}[scale=1]
\draw (0,0) rectangle (1,1);
\draw (3/4,1) to [out=-90,in=90] (1/4,0);
\draw (1/4,1) arc (180:360:1/8);
\draw (1/2,0) arc (180:0:1/8);
\end{tikzpicture} = d \begin{tikzpicture}[scale=1]
\draw (0,0) rectangle (1,1);
\draw (1/4,1) arc (180:360:1/8);
\draw (1/4,0) arc (180:0:1/8);
\draw (3/4,0) --(3/4,1);
\end{tikzpicture}  \]

There is a natural linear anti-multiplicative involution $D \mapsto D^t$ on  $\C[NC_2(2k)]$ given simply by the linear extension of the operation of turning diagrams upside-down.   For example,
\[  \begin{tikzpicture}[scale=1]
\draw (0,0) rectangle (1,1);
\draw (1/4,1) to [out=-90,in=90] (3/4,0);
\draw (1/2,1) arc (180:360:1/8);
\draw (1/4,0) arc (180:0:1/8);
\end{tikzpicture}^t = 
\begin{tikzpicture}[scale=1]
\draw (0,0) rectangle (1,1);
\draw (3/4,1) to [out=-90,in=90] (1/4,0);
\draw (1/4,1) arc (180:360:1/8);
\draw (1/2,0) arc (180:0:1/8);
\end{tikzpicture}\]

It is well known that the above algebraic structure on $\C[NC_2(2k)]$ is isomorphic to  the Temperley-Lieb algebra $\TL_k(d)$ when $d \in \C \backslash \{2 \cos(\frac \pi n)\}_{n=2}^{k+1}$, the isomorphism being given in terms of the generators $1, u_1, \ldots, u_{k-1} \in \TL_k(d)$ by 
\[1 \mapsto \begin{tikzpicture}
\draw (0,0) rectangle (1,1);
\draw (1/16,0) -- (1/16,1);
\draw (3/16,0) node{\tiny $\cdots$};
\draw(5/16, 0) -- (5/16, 1);

\draw (1/2,0) node[below] {};
\draw (1.5,1/4) node[below] {};

\draw (11/16,0) -- (11/16,1);
\draw (13/16, 1/2) node{\tiny$\cdots$};
\draw (15/16,0) --(15/16,1);
\end{tikzpicture}, \qquad u_i \mapsto 
\begin{tikzpicture}
\draw (0,0) rectangle (1,1);
\draw (1/16,0) -- (1/16,1);
\draw (3/16,0) node{\tiny $\cdots$};
\draw(5/16, 0) -- (5/16, 1);
\draw (3/8,1) arc (180:360:1/8);

\draw (1/2,0) node[below] {};
\draw (1.5,1/4) node[below] {};

\draw (3/8,0) arc (180:0:1/8);
\draw (11/16,0) -- (11/16,1);
\draw (13/16, 1/2) node{\tiny$\cdots$};
\draw (15/16,0) --(15/16,1);
\end{tikzpicture}\] 
See for example \cite{Ka87, KaLi94, CaFlSa95}.  As a result, from now on we shall identify these two algebras as the same object.

\subsection{The Markov trace}

The {\it Markov trace} is the tracial linear functional $\Tr:\TL_k(d) \mapsto \C$ that sends a diagram $D \in \TL_k(d)$ to the following complex number, called the {\it tracial closure of $D$}:

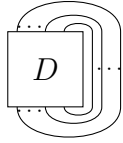
$$\begin{tikzpicture}
\node (diagram) {};
\draw (0,0) rectangle (1,1);
\draw (1/2,1/2) node {$D$};
\draw  (3/4,1) to [out=90, in=90] (9/8,1)--(9/8,0) to [out=-90,in=-90] (3/4,0);
\draw  (1/2,1) to [out=90, in=90] (10/8,1)--(10/8,0) to [out=-90,in=-90] (1/2,0);
\draw (5/16,-1/16) node {\tiny $\cdots$};
\draw (5/16,17/16) node {\tiny $\cdots$};
\draw (1/8,1)  to [out=90, in=90] (12/8,1)--(12/8,0) to [out=-90,in=-90] (1/8,0);
\draw (11/8,1/2) node {\tiny $\cdots$};
\end{tikzpicture}$$
In other words, we  connect the $k$ points on the top of $D$ to the $k$ points on the bottom of $D$ as indicated in the above picture.  The result is a system of loops in the plane.  The number of resulting loops is denoted by $\#\text{loops}(D)$, and then we have 
\[
\Tr(D) = d^{\#\text{loops}(D)}.
\]
Using the Markov trace and the transpose $t$, we can define a symmetric bilinear pairing $\langle \cdot, \cdot \rangle: \TL_k(d) \times \TL_k(d) \to \C$ given by
\[\langle D, D' \rangle = \Tr(D^tD') \qquad (D,D' \in \TL_k(d)).\]
This bilinear form turns out to be non-degenerate precisely when $\TL_k(d)$ is semisimple, and this is guaranteed to happen when $d \ne 2\cos(\frac \pi n)$ for $n \ne  2,3,4, \ldots, k+1$.  See for example \cite{We87, Li91, BaCu10}.  For the remainder of the paper, we make the assumption that the the bilinear form $\langle \cdot, \cdot \rangle$ defined above is non-degenerate.  

\begin{remark}For future reference, we also note at this time the following well-known combinatorial formulas for the Markov trace of basic diagrams $D_p$, $p \in NC_2(2k)$:
\[\Tr(D_p) = d^{|p \vee \bf 1|}, \quad \langle D_p, D_q \rangle = \Tr(D_p^tD_q) = d^{|p \vee q|} \qquad (p,q \in NC_2(2k)). \]
In the above formula, ${\bf 1} = \{\{1, 2k\}, \{2, 2k-1\}, \ldots , \{k, k+1\}\} \in NC_2(2k)$ denotes the ``identity'' partition (corresponding the the identity $1 = D_{\bf 1} \in \TL_k(d)$). These identities are easily verified by the reader.

\end{remark}

\subsection{The Jones-Wenzl projections}
We now come to one of the main objects of study in this paper.

\begin{definition}
Let $k \in \N$ and $d \in \C \backslash \{2\cos(\frac \pi n)\}_{2 \le n \le k+1}$  be as above.  Then there exists a unique non-zero idempotent $q_k \in \TL_k(d)$, called the {\it Jones-Wenzl projection}, with the property that
\begin{align}\label{JW-condition}
u_iq_k = q_ku_i = 0 \qquad (i=1, \ldots, k-1). 
\end{align}
\end{definition}

Although the above defining relations for $q_k$ are simple to state, they are not very useful for determining the structure of $q_k$.  For determining the decomposition of $q_k$
as a linear combination of basis diagrams $\{D_p\}_{p \in NC_2(2k)}$,  we instead have the {\it Wenzl recursion formula}.  In what follows, we use the notation
$$\begin{tikzpicture}
\node (first) {};
\draw (-1/2,1/2) node {$q_{k}=$};
\draw (1/4,1/2) rectangle (1/2,3/4);
\draw (3/8, 0) -- (3/8, 2/4);
\draw (3/8, 3/4) -- (3/8, 5/4);
\draw (5/8, 1/4) node {$k$};
\draw (1,1/2) node {$=$};
\node (second) [right=of first] {};
\begin{scope}[shift={(second)}]
\draw (1/4, 0) -- (1/4, 5/4);
\draw (4/8, 1/4) node {$k$};
\end{scope}
\end{tikzpicture}$$
 to represent  $q_k$, then the Wenzl recursion \cite{We87} is given by $q_1=$
\begin{tikzpicture}[scale=.5]
\draw (0,0) rectangle (1,1);
\draw (1/2,0)--(1/2,1);
\end{tikzpicture}, $q_2=$ \begin{tikzpicture}[scale=.5]
\draw (0,0) rectangle (1,1);
\draw (3/4,0) --(3/4,1);
\draw (1/4,0) --(1/4,1);
\end{tikzpicture} $-\frac{1}{d}$
\begin{tikzpicture}[scale=.5]
\draw (0,0) rectangle (1,1);
\draw (1/4,1) arc (180:360:1/4);
\draw (1/4,0) arc (180:0:1/4);
\end{tikzpicture} and $q_k$ is given inductively by

$$\begin{tikzpicture}
\node (first) {};
\draw (-1/2,1/2) node {$q_{k+1}=$};
\draw (1/4,1/2) rectangle (1/2,3/4);
\draw (3/8, 0) -- (3/8, 2/4);
\draw (3/8, 3/4) -- (3/8, 5/4);
\draw (5/8, 1/4) node {$k$};
\node (second) [right= of first,xshift=-.25cm] {};
\begin{scope}[shift={(second)}]
\draw (1/2,1/2) node {$- $ $\frac{\Delta_{k-1}(d)}{\Delta_k(d)}$};
\draw (3/2, 3/4) rectangle (3,1);
\draw (3/2, 1/4) rectangle (3,1/2);
\draw (2, 0)--(2, 5/4);
\draw (7/4,0)--(7/4,5/4);
\draw (9/4,5/8) node {\tiny $\cdots$};
\draw (10/4,0)--(10/4,5/4);
\draw (11/4,5/4)--(11/4,13/16) to [out=-90,in=-90] (13/4,13/16)--(13/4,5/4);
\draw (11/4,0)--(11/4,7/16) to [out=90,in=90] (13/4,7/16)--(13/4, 0);
\end{scope}
\end{tikzpicture}.$$
Here, $\Delta_k$ is the $k$th Chebyshev polynomial of type 2, defined by 
\[
\Delta_0(x) = 1, \quad \Delta_1(x) = x, \quad x\Delta_k(x) = \Delta_{k+1}(x) + \Delta_{k-1}(x) \quad (k \ge 1).
\]  In fact, $\Delta_k(d) = \Tr(q_k)$ is the Markov trace of the $k$th Jones-Wenzl projection in $\TL_k(d)$. It is relatively straightforward to check that this construction results in idempotent objects that annihilate the generators $u_1, \ldots, u_{k-1}$ of $\TL_k(d)$.

\subsection{Jones-Wenzl projections and the dual diagram basis} Given a finite-dimensional vector space $E$ equipped with a non-degenerate bilinear form $\langle \cdot , \cdot \rangle$ and a linear basis $\mc B=\{x_1,\ldots ,x_n\}$ for $E$, recall that the {\it dual basis} associated to $\mc B$ is the unique linear basis $\hat {\mc B}=\{\hat x_1,\ldots , \hat x_n\}$ of $E$ with the property that \[\langle x_i,\hat x_j\rangle =\delta_{ij}. \qquad (1 \le i,j \le n).\]

For $\TL_k(d)$ (with $k \in \N, d \in \C \backslash \{2\cos(\frac \pi n)\}_{2 \le n \le k+1}$), equipped with its non-degenerate bilinear form induced by the  Markov trace, we consider the canonical Temperley-Lieb diagram basis $\mc B = \{D_p\}_{p\in NC_2(2k)}$ and the corresponding dual basis $\hat {\mc B} = \{\hat D_p\}_{p \in NC_2(2k)}$. 

 In terms of the diagram basis and its dual, we have the following (presumably well-known) lemma relating the Jones-Wenzl projections to certain dual basis elements.

\begin{lemma} \label{JW-dualbasis}
Let  $k \in \N$ and  $d \in \C \backslash \{2\cos(\frac \pi n)\}_{2 \le n \le k+1}$.  Then the $k$th Jones-Wenzl projection $q_k \in \TL_k(d)$ is given by 
\[
q_k = \frac{\hat D_{\bf 1}}{\langle \hat D_{\bf 1},\hat D_{\bf 1}\rangle},
\]
where, as before, ${\bf 1} = \{\{1, 2k\}, \{2, 2k-1\}, \ldots , \{k, k+1\}\} \in NC_2(2k)$.
\end{lemma}

\begin{proof}
We first observe that the coefficient of $D_{\bf 1}$ appearing in the expansion of $q_k$ in terms of the diagram basis $\mc B$ is always 1.  Indeed, from the definition of $q_k$, we have $q_k I_k = I_kq_k = \{0\}$, where $I_k \lhd \TL_k(d)$ is the codimension 1 ideal generated by $u_1, \ldots, u_{k-1}$.  In particular, we can uniquely write $q_k = \alpha_k D_{\bf 1} + g_k$, where $\alpha_k \in \C$ and $g_k \in I_k$.  But then we have 
\[
q_k = q_k^2 = (\alpha_k D_{\bf 1} + g_k )q_k = \alpha_k q_k,  
\] 
which forces $\alpha_k = 1$.

Next, we observe that for any ${\bf 1} \ne p \in NC_2(2k)$, we have 
\[
\langle D_p, q_k \rangle = \Tr(D_p^tq_k) = \Tr(0) = 0 \qquad (\text{since $D_p, D_p^t \in I_k$}).
\]
Therefore there exists a $c_k \in \C$ such that $q_k = c_k \hat D_{\bf 1}$.  Moreover,
\[
c_k\langle \hat D_{\bf 1},\hat D_{\bf 1}\rangle = \langle \hat D_{\bf 1} , q_k \rangle  = \langle \hat D_{\bf 1}, D_{\bf 1} + g_k\rangle = \langle \hat D_{\bf 1}, D_{\bf 1} \rangle + 0= 1. 
\]
\end{proof}

\section{Free orthogonal quantum groups and the Weingarten calculus} \label{fqg}

\subsection{Free orthogonal quantum groups}

In this section we recall the definition of the free orthogonal quantum groups, introduced by Van Daele and Wang in \cite{VaWa96}.  

\begin{notation}
Given a unital complex $\ast$-algebra $A$ and a matrix $X = [x_{ij}] \in M_n(A)$ ($n \in \N$), we write $\bar X = [x_{ij}^*] \in M_n(A)$, $X^* = [x_{ji}^*]$, and $AXB = [\sum_{k,l=1}^n a_{ik}x_{kl}b_{kj}]$ for any $A = [a_{ij}], \ B = [b_{ij}] \in M_n(\C)$.  We call $X$ {\it unitary} if $X^*X = XX^* = 1 \in M_n(A)$, where $M_n(A)$ is equipped with its usual unital $\ast$-algebra structure inherited from $A$.
\end{notation}

\begin{definition}[\cite{VaWa96}] Fix an integer $n \ge 2$ and $F \in \text{GL}_n(\C)$ such that $F \bar F = \pm 1$. The  {\it algebra of polynomial functions on the free orthogonal quantum group} is the universal unital $\ast$-algebra with generators and relations given by 
\begin{align} \label{eqn:defining} \mc O(O^+_F) := \ast-\text{alg}\big((u_{ij})_{1 \le i,j \le n} \ | \ U = [u_{ij}] \text{ is unitary in $M_n(\mc O(O_F^+))$}\  \& \ U = F \bar U F^{-1}\big).\end{align} 
\end{definition}

\begin{remark}
As the above terminology suggests, the algebras $\mc O(O^+_F)$ are in fact a class of Hopf $\ast$-algebras associated to operator algebraic compact quantum groups in the sense of Woronowicz \cite{Wo87b, Wo98}.    In particular, when $F \in \text{GL}_n(\C)$ is the identity matrix, we write $O^+_n$ instead of $O^+_F$,  and $\mc O(O^+_n)$ is exactly the free non-commutative analogue of the Hopf $\ast$-algebra $\mc O(O_n)$ of polynomial functions on the classical orthogonal group.   Since we will need very little quantum group technology in what follows, we refer the reader to the above references for more details.
\end{remark}

A fundamental feature of the $\ast$-algebra $\mc O(O^+_F)$ is the existence of a {\it coproduct}, which is a unital $\ast$-homomorphism $\Delta:\mc O(O^+_F) \to \mc O(O^+_F) \otimes \mc O(O_F^+)$ determined by 
\[\Delta(u_{ij}) = \sum_{k=1}^n u_{ik} \otimes u_{kj} \qquad (1 \le i,j \le n),\] and satisfying the {\it co-associativity} relation $(\iota \otimes \Delta)\Delta = (\Delta \otimes \iota)\Delta$.  For the non-commutative algebras $\mc O(O^+_F)$, $\Delta$ plays the role of the group law on the underlying ``quantum space'' $O^+_F$.  It then follows from general theory of compact quantum groups \cite{Wo98} that there exists a unique {\it Haar integral}.  That is, a faithful state $\mu = \mu_F:\mc O(O^+_F) \to \C$ satisfying the left and right invariance condition
\begin{align} \label{Haar-condition}(\mu \otimes \iota)\Delta = (\iota \otimes \mu)\Delta = \mu(\cdot)1.
\end{align}
The Haar integral on $\mc O^+_F$ is a non-commutative generalization of the Haar measure on its classical counterpart $O_n$.

\subsection{Weingarten calculus}

For the computation of moments of the the generators \[\{u_{ij}\}_{1 \le i,j \le n} \subseteq \mc O(O^+_F)\] with respect to the Haar integral, we rely on a combinatorial tool, known as the {\it Weingarten calculus}.  The setup is as follows.

Fix a parameter $d \in \C^*$, and for each $k \ge 1$, define a matrix $G_{d} = [G_{d}(p,q)]_{p,q \in NC_2(2k)}$, by
\[ 
G_{d}(p,q) =d^{|p \vee q|}.
\] 
It is then known that $G_{d}$ is an invertible matrix \cite{Ba96, Ba97, BaCo07, BrKi16, DiFr98} for all such $k \ge 1$ and $d \in \C \backslash \{2\cos(\frac \pi n)\}_{2 \le n \le k+1}$. Denote the inverse matrix $G_{d}^{-1}$ by $\Wg_{d} = [\Wg_{d}(p,q)]_{p,q \in NC_2(2k)}$.  $\Wg_{d}$ is called the {\it  Weingarten matrix (of order $2k$)}, and  any function of the form $d \mapsto \Wg_d(p,q)$ with $p,q \in NC_2(2k)$ is called a {\it Weingarten function}.   The following theorem shows that the Weingarten functions encode all of the data of the moments of the standard generators of $\mc O(O^+_F)$.

\begin{theorem}[\cite{BaCo07, BaCoZi09,  BrKi16}] \label{thm:moments} 
Let $n \ge 2$, $c \in \{\pm 1\}$ and let $F \in GL_n(\C)$ be such that $F\bar F = c1$.  Set  $d :=c\Tr(F^*F)$ and consider the generators $\{u_{ij}\}_{1 \le i,j \le n}$ of the free orthogonal Hopf $\ast$-algebra $\mc O(O^+_F)$.   For each $l \in \N$ and each pair of multi-indices $i,j:[l] \to [n]$, we have 
\[\mu(u_{i(1)j(1)}u_{i(2)j(2)}\ldots u_{i(l)j(l)}) = 0\] if $l$ is odd, and otherwise 
\[\mu(u_{i(1)j(1)}u_{i(2)j(2)}\ldots u_{i(l)j(l)}) = \sum_{p,q \in NC_2(l)} c^{\frac l 2} \Wg_{d}(p,q) \overline{\delta^F_p(j)} \delta^F_q(i), \]
where \[\delta^F_p(j) = \prod_{\{s,t\} \in p} F_{j(t)j(s)} \quad \& \quad \delta^F_p(i) = \prod_{\{s,t\} \in q} F_{i(t)i(s)},\]
and $\{s,t\} \in p$ (resp. $\{s,t\} \in p$) means that $\{s,t\}$ is a block of $p$ (resp. a block of $q$).
\end{theorem}

\begin{remark}
In what follows, we will mainly be interested in two types of matrices $F \in \text{GL}_n(\C)$.  The first type is the identity matrix $F = 1 \in \text{GL}_n(\C)$, and the second is of the form 
\[F_\rho = \left(\begin{matrix}1_{n-2} & 0 & 0 \\
0 & 0 &\rho \\
0&\rho^{-1} & 0
\end{matrix} \right) \in \text{GL}_n(\C),
\] 
where $0 < \rho < 1$.

In both of these cases, the Weingarten formula of Theorem \ref{thm:moments} is readily seen to simplify to
\[\mu(u_{i(1)j(1)}u_{i(2)j(2)}\ldots u_{i(l)j(l)}) = 0\] if $l$ is odd, and 
\[\mu(u_{i(1)j(1)}u_{i(2)j(2)}\ldots u_{i(l)j(l)}) = \sum_{\substack{p,q \in NC_2(l)\\
\ker j \ge p, \ \ker i \ge q}} \Wg_{d}(p,q) \qquad ( \forall \ i,j :[l]\to [n-2] ).\] 
In the above formula, $\ker i$ is the partition of $[l]$ determined by the condition that $s,t$ belong to the same block of $\ker i$ if and only if $i(s) = i(t)$, and as mentioned before, the symbol $\ge$ denotes the refinement ordering on partitions of $[l]$.   
\end{remark}

Finally, let us introduce a notion of {\it generic monomials}, that will allow in some cases to reformulate conveniently certain statements and formulas in this paper, and  describe an explicit link between Haar integration over quantum groups and coefficients of the dual diagram basis in certain Temperley-Lieb algebras. 
Let $F=1$ or $F = F_\rho$ as in the preceding remark and fix $p,q\in NC_2(2k)$.  We shall call a degree $2k$ monomial of the form $u_{i(1)j(1)}\ldots u_{i(2k)j(2k)} \in \mc O(O^+_F)$  a {\it $(p,q)$-generic monomial} if the only partition in $NC_2(2k)$ is finer than $\ker j$ (respectively $\ker i$) is $p$ (respectively $q$).  Any such $(p,q)$-generic monomials will be denoted by $u_{p,q}$.  In particular, 
$\mu (u_{p,q})=\Wg _{d}(p,q)$ in the regimes described in the previous remark. 

It is easy to see that for any $p,q\in NC_2(2k)$, a $(p,q)$-generic monomial exists provided that $n-2 \ge k$. 

\subsection{Weingarten calculus and the dual diagram basis in $\TL_k(d)$} \label{mainresult}

We are now ready to establish the main result of this section, setting a link between the Haar measure over free
orthogonal quantum groups and the dual diagram basis for the Temperley-Lieb algebra.  We recall that $\{D_p\}_{p \in NC_2(2k)}$ denotes the Temperley-Lieb diagram basis for $\TL_k(d)$, and $\{\hat D_p\}_{p \in NC_2(2k)}$ denotes the dual basis.

Fix $k \in \N$ and $d \in [2,\infty)$.  Choose $n = n(d) \in \N$ and $F = F(d) \in \text{GL}_n(\C)$ so that 
\[F(d) = 1 \in  \text{GL}_d(\C)\]
if $d \in \N$, and otherwise \[F(d) = F_\rho = \left(\begin{matrix}1_{n-2} & 0 & 0 \\
0 & 0 &\rho \\
0&\rho^{-1} & 0
\end{matrix} \right) \in \text{GL}_n(\C),
\]  
where $0 < \rho < 1$ is chosen so that $d = \Tr(F^*F) =  n-2 + \rho^2 + \rho^{-2}$.
\begin{theorem} \label{main}
With the notations fixed as above, the dual basis element $\hat D_p$ associated to a diagram $D_p \in \TL_k(d)$ is given by  
\[\hat D_p =\sum_{q\in NC_2(2k)} \Wg _{d}(p,q)D_q = \sum_{q\in NC_2(2k)} \mu_{F(d)} (u_{p,q})D_q \qquad (p \in NC_2(2k)),  \]
where the first equality holds for $d \in \C \backslash \{2 \cos(\frac \pi n)\}_{n=2}^{k+1}$ and the second equality holds at least for $d \ge k$. 
\end{theorem}

\begin{proof}
The first equality follows from the fact that the transfer matrix from a basis to its dual basis is the inverse of the Gram matrix 
of the basis. This matrix is exactly the Weingarten matrix in our case. Indeed, the Gram matrix for the basis $\{D_p\}_{p \in NC_{2k}}$ has coefficients $\langle  D_p , D_q\rangle = \Tr(D_q^t D_p) = d^{|p \vee q|} = G_{d}(p,q)$.  For the second equality, we just observe that the condition $d \ge k$ implies that $n-2 \ge k$, and this implies the existence of $(p,q)$-generic monomials for all $p,q \in NC_2(2k)$.
\end{proof}

As for the Jones-Wenzl projections, we have the following consequence

\begin{theorem} \label{thm:coeffs-Weingarten-JW}
The $k$th Jones-Wenzl projection $q_k \in \TL_k(d)$ is given by 
\[q_k  = \sum_{q\in NC_2(2k)} \frac{\Wg _{d}({\bf 1},q)}{\Wg _{d}({\bf 1},{\bf 1})}D_q = \sum_{q\in NC_2(2k)} \frac{\mu_{F(d)}(u_{{\bf 1},q})}{\mu_{F(d)}(u_{{\bf 1},{\bf 1}})}D_q,\]
where the first equality holds for $d \in \C \backslash \{2 \cos(\frac \pi n)\}_{n=2}^{k+1}$ and the second equality holds at least for $d \ge k$.  
\end{theorem}

\begin{proof}
In view of Lemma \ref{JW-dualbasis} and Theorem \ref{main}, all we need to show is that $\langle \hat D_{\bf 1},\hat D_{\bf 1}\rangle = \Wg_{d}({\bf 1}, {\bf 1})$.  But this just follows from the calculation
\begin{align*}\langle \hat D_{\bf 1},\hat D_{\bf 1}\rangle &= \sum_{q,q' \in NC_2(2k)} \Wg_{d}({\bf 1}, q) \Wg_{d}({\bf 1}, q') \langle D_q, D_{q'} \rangle \\
&=  \sum_{q' \in NC_2(2k)} \delta_{{\bf 1}q'} \Wg_{d}({\bf 1}, q')  =  \Wg_{d}({\bf 1},{\bf 1}) .  
\end{align*}
\end{proof}

\section{The Laurent series expansion for $\Wg_{d}(p,q)$ and applications to the dual $\TL_k(d)$ basis} \label{app}

We now come to the main section of the paper, where we study the structure of the Laurent series expansions of the free orthogonal Weingarten functions $d \mapsto \Wg_d(p,q)$ ($p,q \in NC_2(2k)$) on the annuli $|d|> 2\cos\big(\frac{\pi}{k+1}\big)$.  Our first goal here is to give more precise formulations of Theorems A and B (see Theorems \ref{thm:Laurent-series}--\ref{thm:non-zero-coefficients}) with the aid of some graph theoretical tools, followed a presentation of the proofs of these results.  

We begin by introducing some technical tools  tools  that will be used to formulate and prove Theorem \ref{thm:Laurent-series}.

\subsection{Non-crossing neighbors in $NC_2(2k)$}

\begin{definition} \label{nc-neighbor}
Fix $k \ge 2$.  Given two non-crossing  partitions $p \ne p' \in NC_2(2k)$, we say that $p'$ is a  {\it non-crossing neighbor} of $p$ (denoted by $p \to p'$), if there exists an interval $\{t,t+1\} \in p$ and another pair $\{x,y\}\in p$ with the property that
\begin{enumerate}\item The partition \[p'' = \{t,t+1, x,y\}\cup \bigcup_{\{r,s\} \ne \{t,t+1\}, \{x,y\} }\{r,s\} \text{ is non-crossing},\] and 
\item $p' \le p''$ is the unique element of $NC_2(2k)$ such that $p' \ne p$.
\end{enumerate}
In other words, we have $p \to p'$ if and only if $p'$ can be obtained from $p$ by joining an interval of $p$ to another pair in $p$ to produce a non-crossing partition $p''$ for which $p' \le p''$. \end{definition}

\begin{remark}
It is important to note that the above definition is not symmetric.  I.e., $p \to p'$ does not necessarily imply $p' \to p$.  Take for example $p = \{1,2\}\{3,4\}\{5,6\}$ and $p' = \{1,6\}\{2,5\}\{3,4\}$.
\end{remark}

Our reason for considering the above notion of non-crossing neighbors is that it is intimately connected to certain algebraic relations between Haar integrals of generic monomials (or equivalently Weingarten functions) over the free orthogonal quantum groups $O^+_d$ ($d \in \N$).  Indeed, suppose $p,q \in NC_2(2k)$ with $k \le d \in \N$ and consider a $(p,q)$-generic monomial $u_{p,q} = u_{i(1)j(1)}u_{i(2)j(2)}\ldots u_{i(2k)j(2k)} \in \mc O(O^+_d)$, where $i,j:[2k]\to [d]$ are fixed multi-indices satisfying $\ker i = p$ and $\ker j = q$.  Now fix an interval interval $\{t,t+1\} \in p$ and assume without loss of generality that $i(t) = i(t+1)= 1$.  Using the defining orthogonality relations \[\sum_{s=1}^d u_{sr}u_{sr'}=\delta_{rr'}1 = \sum_{s=1}^d u_{rs}u_{r's} \qquad (1 \le r,r' \le d)\] for the generators of $\mc O(O^+_d)$, we obtain the relation
\begin{align*}
&\sum_{i(t)  =i(t+1)=1}^d u_{i(1)j(1)}u_{i(2)j(2)}\ldots u_{i(2k)j(2k)} = u_{p,q} + \sum_{i(t)  =i(t+1)=2}^d u_{i(1)j(1)}u_{i(2)j(2)}\ldots u_{i(2k)j(2k)}\\
&=  \delta_{j(t)j(t+1)}u_{\tilde p_t \tilde q_t}= \delta_{\{t,t+1\} \in q}u_{\tilde p_t \tilde q_t},
\end{align*}
where $\tilde p_t, \tilde q_t \in NC_2(2k-2)$ are the natural partitions obtained from $p,q$ by removing the common interval $\{t,t+1\}$.  Integrating this relation over $O^+_d$ and using the Weingarten formula, we obtain
\begin{align*}&\mu(u_{p,q}) + \sum_{i(t)  =i(t+1)=2}^d \mu(u_{i(1)j(1)}u_{i(2)j(2)}\ldots u_{i(2k)j(2k)}) 
= \delta_{\{t,t+1\} \in q}\mu(u_{\tilde p_t \tilde q_t}) \\
\Leftrightarrow &\Wg_d(p,q) + \sum_{i(t)  =i(t+1)=2}^d \sum_{\substack{p'_t \in NC_2(2k)\\ \ker i \ge p'_t}}\Wg_d(p_t',q)
= \delta_{\{t,t+1\} \in q}\Wg_d(\tilde p_t ,\tilde q_t).
\end{align*}
But for $2 \le i(t) = i(t+1) \le d$, $p'_t \in NC_2(2k)$ is easily seen to satisfy $\ker i \ge p'_t$ if and only if either $p'_t = p$ or $p'_t$ is a non-crossing neighbor of $p$.   In particular,  we have the following relation
\begin{align}\label{Wg-relation1} d \Wg_d (p,q) + \sum_{\substack{p'_t \in  NC_2(2k)\\
p \to p'_t}}\Wg_d(p'_t,q) = \delta_{\{t,t+1\} \in q}\Wg_d(\tilde p_t, \tilde q_t)  \qquad (p,q \in NC_2(2k)),
\end{align}
Of course there is an obvious analogue of equation \eqref{Wg-relation1} where the summation occurs over the second variable in $\Wg_d$ instead of the first variable:  
\begin{align}\label{Wg-relation2} d \Wg_d(p,q) + \sum_{\substack{q'_t \in  NC_2(2k)\\
q \to q'_t}}\Wg_d(p,q'_t) = \delta_{\{t,t+1\} \in p}\Wg_d(\tilde p_t, \tilde q_t)  \qquad (p,q \in NC_2(2k)),
\end{align}
Equations \eqref{Wg-relation1}--\eqref{Wg-relation2} will be of crucial importance in what follows, and we shall refer to them as the {\it Weingarten orthogonality relations} associated to an interval $\{t,t+1\}$ belonging to one of the pairings $p$ or $q$.

\begin{example} \label{relations-ex}
It is perhaps worthwhile to clarify the above Weingarten orthogonality relations with a concrete example.  Let $k = 4$, $p = \{1,2\}\{3,4\}$ and $q =\{1,4\}\{2,3\}$.  Then the orthogonality relation associated to the interval $\{1,2\} \in p$  gives the Weingarten orthogonality relation
\begin{align*}
d\Wg_d(p,q) +\Wg_d(q,q) = 0,
\end{align*}
since in this case the only non-crossing neighbor of $p$ is $q$,  and $\{1,2\}\notin q$ which explains why the right hand side is zero. 

As a motivation for the ideas to come, let us suppose we are interested in evaluating $\Wg_d(p,q)$.  Noting that the above Weingarten orthogonality relation expresses $\Wg_d(p,q)$ in terms of $\Wg_d(q,q)$, this suggests we next consider a Weingarten orthogonality relation associated to the pair $(q,q)$.  Taking the orthogonality relation associated to the interval $\{1,2\} \in q$, we get
\[
d \Wg_d(q,q) + \Wg_d(p,q) = \Wg_d(\tilde q, \tilde q) = d^{-1},
\]
where $\tilde q  \in NC_2(2)$ is the unique pairing obtained by deleting the interval $\{1,2\}$ from $q$. 
We thus obtain two equations in two unknowns which can readily  be solved to obtain 
\[\Wg_d(p,q) =\frac{-1}{d^3-d} \quad \& \quad \Wg_d(q,q) = \frac{1}{d^2-1}. \] 
This informally suggests that sequences of judiciously chosen Weingarten orthogonality relations allows one to solve a system of equations to evaluate general Weingarten functions $\Wg_d(p,q)$.
\end{example}

\begin{remark} \label{Wg-credit}The Weingarten orthogonality relations \eqref{Wg-relation1}--\eqref{Wg-relation2} are special examples of the relations that Weingarten initially used in his study of the large $d$-asymptotics of polynomial integrals over the unitary groups $U_d$  \cite{We78}.   What is nowadays called the ``Weingarten calculus'' for compact (quantum) groups focuses on analyzing the Weingarten function directly, without direct reference to the underlying orthogonality/unitarity relations, and with more powerful and conceptual tools such as representation theory,  combinatorics, etc.  However, as we shall see, the present paper shows that getting back to the defining orthogonality relations for the quantum groups $O_d^+$ at hand turns out to yield the strongest results available.
\end{remark}

\subsection{The Weingarten Graph}
We now define a directed graph structure on pairs of non-crossing pairings which is designed to help keep track of what kinds of new non-crossing pairings arise when considering the Weingarten orthogonality relations  \eqref{Wg-relation1}--\eqref{Wg-relation2}.

\begin{definition}
We define an infinite directed graph $\mc G = (V_\mc G,E_\mc G)$ as follows.  The vertex set is given by 
\[
V_\mc G = \bigsqcup_{k \in \N_0} NC_2(2k) \times NC_2(2k),
\]
where by convention we define $NC_2(0) \times NC_2(0) = \{(\emptyset, \emptyset)\}$.
The set of directed edges $E_\mc G \subset V_\mc G \times V_\mc G$ given by the following two rules. 
\begin{enumerate}
\item If  $p,q, p', q' \in NC_2(2k)$, then $((p,q), (p',q')) \in E_\mc G$ if and only if 
\begin{enumerate}
\item $p \to p'$ and $q = q'$, or
\item $q \to q'$ and $p  = p'$
\end{enumerate}
\item If $p, q \in NC_2(2k)$ and $p', q' \in NC_2(2k-2)$, then $((p,q), (p', q')) \in E_\mc G$ if and only if there exists a common interval $\{t,t+1\} \in p,q$ and $p',q'$ are the pairings obtained from $p,q$ by removing this common interval.
\end{enumerate}
\end{definition}

We call $\mc G$ the  {\it Weingarten graph}.  As mentioned above, the edge set $E_\mc G$ is constructed exactly to encode all pairs of non-crossing pairings that might arise in the Weingarten orthogonality relations \eqref{Wg-relation1}--\eqref{Wg-relation2}.  In particular, using the structure of $\mc G$, we can succinctly rewrite the two Weingarten orthogonality relations \eqref{Wg-relation1}--\eqref{Wg-relation2} associated to an interval $\{t,t+1\}$ (belonging to at least one of $p,q \in NC_2(2k)$) as 

\begin{align} \label{E-meaning}
d\Wg_d(p,q) + \sum_{((p,q),(p_1,q_1))} \Wg_d(p_1,q_1)  = \delta_{\{t,t+1\} \in p}\delta_{\{t,t+1\} \in q}\Wg_d(\tilde{p}_1, \tilde{q}_1).
\end{align}
Here, the sum above runs over all edges $((p,q),(p_1,q_1)) \in E_\mc G$ with $p_1,q_1 \in NC_2(2k)$ which appear as a result of the chosen $O^+_d$-orthogonality relation taken at $\{t,t+1\}$, and  the $\Wg_d(\tilde{p}_1, \tilde{q}_1)$ term corresponds to the edge $((p,q),(\tilde{p}_1,\tilde{q}_1))  \in E_\mc G$ that exists if and only if $p$ and $q$ share $\{t,t+1\}$ as a common interval.  

Of course, not {\it all} edges in $E_\mc G$ with source $(p,q)$ appear in the above sum, just those associated the particular choice of interval $\{t,t+1\}$  made in \eqref{E-meaning}.  We will address this issue again in Section \ref{sgr}, but first we derive some useful basic properties of the Weingarten graph $\mc G$.

\begin{proposition} \label{prop:connected}
For every pair $(p,q) \in NC_2(2k) \times NC_2(2k)$, there exists a (generally non-unique) directed path in the Weingarten graph $\mc G$ connecting the vertex $(p,q)$ to the vertex $(\emptyset,\emptyset)$. 
\end{proposition}

\begin{proof}
Denote by $\Delta(\mc G)$ the {\it diagonal} of $V_\mc G$. I.e., $\Delta(\mc G) = \bigsqcup_{k \in \N_0} \{(p,p): p \in NC_2(2k)\} \subset V_{\mc G}$.  We first observe if $(p,p) \in NC_2(2k) \times NC_2(2k) \subset \Delta(\mc G)$, then there always exists a directed path from $(p,p)$ to $(\emptyset,\emptyset)$ of length $k$.  Indeed one can successively delete intervals from $p$ to build a chain of $k$ edges connecting $(p,p)$ to $(\emptyset,\emptyset)$.

Next, consider a vertex $(p,q) \in V_{\mc G}$, with $p \ne q \in NC_2(2k)$. In view of the observation of the previous paragraph, the proof will be complete if we can show that $(p,q)$ can be connected to $\Delta(\mc G)$ via a directed path. In fact, we will show that such a vertex $(p,q)$ can always be connected to $(p_0,p_0) \in \Delta(\mc G)$, where $p_0 = \{1,2\}\{3,4\}\ldots, \{2k-1,2k\}$ is the interval pairing.  To do this, we will show that $p$ is connected by a sequence of non-crossing neighbors to $p_0$.  The same argument will apply to $q$, and the result will then follow.  To this end, we proceed by the following inductive argument.
\begin{itemize}
\item If $k=1$, $p = p_0$ and there is nothing to prove.
\item Suppose that for some $k_0 \ge 1$, every $p \in NC_2(2k_0)$ is connected by a sequence of non-crossing neighbors to $p_0  \in NC_2(2k_0)$. 
\item Let $k = k_0+1$, $p \in NC_2(2k)$ and let $\{t,t+1\} \in p$ be the rightmost interval when we scan the blocks of $p$ from left to right. 
\item If $t+1 \ne 2k$, let $\{x,t+2\} \in p$ be the block containing $t+2$ and let $p'$ be the non-crossing neighbor of $p$ uniquely determined to have pairs $\{t+1, t+2\}, \{x,t\} \in p'$.  
\item Iterating the previous two steps, we eventually arrive (after $2k-t-1$ iterations) to a new partition $\tilde p \in NC_2(2k)$ of the form \[\tilde p = p_1 \cup \{2k-1,2k\} \quad \text{with} \quad p_1 \in NC_2(2k-2) = NC_2(2k_0). \] 
\item By the induction hypothesis, we can connect $p_1$ to $\{1,2\}\{3,4\} \ldots \{2k_0-1,2k_0\}$ by a sequence of non-crossing neighbors in $NC_2(2k_0)$.  But under the embedding $NC_2(2k_0) \hookrightarrow NC_2(2k);$ $z \mapsto z \cup \{2k-1, 2k\}$, this sequence is easily seen to be a sequence of non-crossing neighbors in $NC_2(2k)$.
\end{itemize}
The proof is then complete by induction.
\end{proof}

Since we now know that every vertex in $\mc G$ is connected to $(\emptyset, \emptyset)$, we can consider directed paths of shortest length: 

\begin{definition}
Given $(p,q) \in  NC_2(2k) \times NC_2(2k) \subset V_\mc G$, we denote by $L(p,q) \in \N_0$ the length of the geodesic ( = shortest directed path) from $(p,q)$ to $(\emptyset,\emptyset)$. 
\end{definition}

\begin{remark}
Of course there are in general many geodesics connecting a given $(p,q)$ to $(\emptyset,\emptyset)$.  As an example, consider the vertex $(p,p) \in V_\mc G$ with $p \in NC_2(2k)$.  Then $L(p,p) = k$, and unless $p$ is the fully nested pairing $\{1,2k\}\{2,2k-1\}\ldots \{k,k+1\}$, there are at least 2 geodesics connecting $(p,p)$ to $(\emptyset,\emptyset)$, since $p$ has at least $2$ intervals.
\end{remark}

The following parity result for edges in $\mc G$ will be crucial in what follows.

\begin{proposition} \label{prop:parity}
For each edge $((p,q), (p',q')) \in E_\mc G$ in the Weingarten graph $\mc G$, we have 
\begin{align}\label{parity}
|p'\vee q'| = |p \vee q| \pm 1.
\end{align}
\end{proposition}

\begin{proof}
There are two cases to consider.

{\bf Case a}: Assume $p,q \in NC_2(2k)$ and $(p',q') \in NC_2(2k-2)$.  Then by definition of  $((p,q), (p',q')) \in E_\mc G$, $(p',q')$ is obtained from $(p,q)$ by removing a common interval pair.  It is then immediate that $|p\vee q| = |p' \vee q'| + 1$.  

{\bf Case b}: Assume that $p,q, p', q' \in NC_2(2k)$ for some $k \ge 2$.  By symmetry, we can then assume without loss of generality that $q = q'$ and $p \to p'$.  I.e., $p' $ is a non-crossing neighbor of $p$.  Moreover, by applying a cyclic rotation to the index set $[2k]$, we may without loss of generality assume that $p = \{1,2\}\{x,y\} \cup p''$ and $p' = \{1,y\}\{2,x\} \cup p''$ for some $p'' \in NC_2([2k]\backslash \{1,2,x,y\}) \cong NC_2(2k-4)$.   (This is possible since such rotations preserve the non-crossing structure, as well as the quantity $|p \vee q|$.)  Note in particular that in this case $x$ must be odd and $y$ must be even, in order for our partitions to be non-crossing.

Consider now the partition $p \vee q$.  There are two possible sub-cases to consider.
\begin{enumerate}
\item {\it The blocks $\{1,2\},\{x,y\} \in p$ are connected by a block of $p\vee q$}.  We begin by recalling that there is a canonical identification of the set of pairings $\mc P_2(2k)$ of $[2k]$ with the subset of permutations $p \in S_{2k}$ having no fixed points in $[2k]$ and satisfying $p^2 = 1$.  Under this identification, note that the blocks of $p \vee q$ correspond to the distinct ``orbits'' \[\mc O_{p,q}(i) = \{i \to p(i) \to qp(i) \to pqp(i)\to\ldots \to i\} \qquad (i \in [2k]),\] associated to alternating applications of $p$ and $q$ to points $i \in [2k]$.  

Let us now consider the single orbit $\mc O_{p,q}(1)$ which, by assumption on $p \vee q$, contains the elements $1,2,x,y$.  Since we have $p(1) = 2$ and $p(x) = y$, this orbit has the following structure:
\[\mc O_{p,q}(1) = \{1 \to^p 2 \to^q \ldots \to^q x \to^p y \to^q \ldots \to^q 1\}.
\]
Replacing $p$ by $p'$, we now have $p'(1) = y$, $p'(2) = x$ and $p' = p$ elsewhere.  In particular,  $\mc O_{p,q}(1)$ gets broken into two orbits $\mc O_{p',q}(1)$ and $\mc O_{p',q}(2)$, containing $1$ and $2$, respectively:
\[
\mc O_{p',q}(1)  = \{1 \to^{p'} y \to^q \ldots \to^q 1\}, \quad \& \quad \mc O_{p',q}(2) = \{ 2 \to^{p'} x \to^{q} \ldots \to^q x\}. 
\]
Since all other orbits are unchanged by replacing $p$ with $p'$, we conclude in this case that 
\[
|p' \vee q'| = |p' \vee q| = |p \vee q| +1.
\]
\item  {\it The blocks $\{1,2\}, \{x,y\} \in p$ belong to distinct blocks of $p\vee q$}.  In this case we have two orbits $\mc O_{p,q}(1)$ and $\mc O_{p,q}(x)$ as follows: 
\[\mc O_{p,q}(1) = \{1 \to 2 \to \ldots \to 1\} \quad \& \quad \mc O_{p,q}(x) = \{ x \to y \to \ldots \to x\}.
\]
Now, if we replace $p$ by $p'$, we connect have  $p'(1) = y$ and $p'(2) = x$ and thus the two orbits collapse to the orbit 
\[
\mc O_{p',q}(1) = \{1 \to y \to \ldots \to x \to 2 \to \ldots \to 1\}. 
\]
Since all other orbits are unchanged by replacing $p$ with $p'$, we have in this case 
\[
|p' \vee q'| = |p' \vee q| = |p \vee q| -1.
\] 
\end{enumerate}
\end{proof}

\begin{remark}
Using Proposition \ref{prop:parity}, it is relatively straightforward to see that we always have the lower bound \[L(p, q) \ge 2k-|p\vee q| \qquad (p,q \in NC_2(2k).\]  Indeed, to travel from $(p,q)$ to $(\emptyset, \emptyset)$ along a path in $\mc G$, one has to succesively delete $k$ common interval pairs from $p$ and $q$.   Each removal of a loop constitutes one edge along this path, and we require at least $k-|p\vee q|$ more edge traversals on this path to form a total of $k$ common intervals, yielding a total path length of at least $k + (k-|p\vee q|) = 2k-|p \vee q|$.  

Based on this observation, it is tempting to guess that we might always have equality in the above inequality.  This in fact turns out to be false in some cases, as can be seen from Example \ref{ex:CS-fail}. 
\end{remark}

The following corollary of Proposition \ref{prop:parity} will be of use in the  proof of Theorem \ref{thm:Laurent-series}.  

\begin{corollary} \label{length-parity}
For each vertex $(p,q) \ne (\emptyset, \emptyset)$ in $\mc G$, the length of any directed path from $(p,q)$ to $(\emptyset, \emptyset)$ has the same parity as $|p \vee q|$.  In particular, the length of any such directed path is of the form $L(p,q) + 2r$ for some $r \in \N_0$. 
\end{corollary}

\begin{proof}
Consider any path $\mc P = (p_0,q_0)(p_1,q_1)\ldots (p_{\ell(\mc P)},q_{\ell(\mc P)})$ in $\mc G$ of length $\ell(\mc P)$, where $(p_0,q_0) =(p,q)$ and $ (p_{\ell(\mc P)},q_{\ell(\mc P)})= (\emptyset, \emptyset)$.    Consider also the function \[f_\mc P: \{0, \ldots, \ell(\mc P)\} \to \N_0; \qquad  f_\mc P(i) = |p_i \vee q_i|, \]
where we define $|\emptyset \vee \emptyset| = 0$.  For any such $\mc P$, we  have 
\begin{align} \label{parity-eqn}
f_\mc P(0) = |p \vee q|, \quad f_\mc P (\ell(\mc P)) = 0, \quad  \& \quad f_\mc P(i+1) = f_\mc P(i) \pm 1 \qquad (0 \le i \le \ell(\mc P) -1),
\end{align}
where the last  set of equalities follows from Proposition \ref{prop:parity}.  In particular, these equalities taken together imply that if $|p \vee q|$ is even (respectively odd) $\ell (\mc P)$ must also be even (respectively odd). 
\end{proof}

\subsection{Weingarten subgraphs} \label{sgr}

In this subsection we describe a connected subgraph $\mc H$ of the Weingarten graph $\mc G$. It has the same vertex set $V_\mc H = V_\mc G$, but fewer edges.  More specifically, to each vertex $(p,q)$ different from $(\emptyset, \emptyset )$, we associate one and only one 
Weingarten orthogonality relation of the form \eqref{Wg-relation1} or \eqref{Wg-relation2} with the property that it involves a
vertex $(p',q')$ such that the following conditions are satisfied:
\begin{enumerate}
\item[($\mc H1$)] \label{H1} $((p,q),(p',q')) \in E_{\mc G}$
\item[($\mc H2$)] \label{H2}$L(p',q')=L(p,q)-1$
\end{enumerate}
By Proposition \ref{prop:connected} it is always possible to make such a choice.   
For each $(p,q)$, the above choice of edge $((p,q),(p',q'))$ generally produces several edges $((p,q),(p_1,q_1))$ according to equations \eqref{Wg-relation1}--\eqref{Wg-relation2} (the edge $((p,q),(p',q'))$ being one of those edges).  The edge set $E_\mc H \subset E_\mc G$  is then defined to be collection of
all edges $((p,q),(p_1,q_1))$ arising in the above discussion. It is clear that $\mc H$ is a subgraph of $\mc G$, and we  call $\mc H$ a {\it Weingarten subgraph} of $\mc G$.  In particular, the Weingarten orthogonality relation for a given vertex $(p,q) \in V_\mc H$ chosen when defining $\mc H$ is expressed concisely in terms of $\mc H$-data as 
\begin{align}\label{HE-meaning}
d\Wg_d(p,q) + \sum_{\substack{((p,q),(p_1,q_1)) \in E_\mc H \\ p_1,q_1 \in NC_2(2k)}}\Wg_d(p_1,q_1)  =  \delta_{\{t,t+1\} \in p}\delta_{\{t,t+1\} \in q}\Wg_d(\tilde{p}_1, \tilde{q}_1),
\end{align}
where $((p,q),(\tilde{p}_1,\tilde{q}_1)) \in E_\mc H $ is the unique edge with $\tilde{p}_1, \tilde{q}_1 \in NC_2(2k-2)$ (if it exists).

It is important to note, however, that the Weingarten subgraph $\mc H$ is not uniquely defined. There are many graphs that fulfill the above definition, depending on the choices of orthogonality relations that are made.  But for the purpose of the forthcoming statements, we need to make a choice.  Surprisingly, the choice of Weingarten subgraph $\mc H$ does not affect our statements or proofs (and to our mind, this is a highly non-trivial fact from a combinatorial point of view, although it follows naturally from our analysis). 

Before stating our main result, we need one more definition in order to describe the Laurent coefficients of $\Wg_d(p,q)$.

\begin{definition} \label{m_r}
Fix a Weingarten subgraph $\mc H \subset \mc G$.  For each vertex $(p,q) \in V_{\mc H}$ and each $r \in \N_0$, we denote by $m_r(p,q)$ the number of directed paths from $(p,q)$ to $(\emptyset, \emptyset)$ of length $L(p,q)+2r$ that are contained in $\mc H$.
\end{definition}

\begin{remark}
A priori, one might expect that the number  $m_r(p,q)$ depends on our choice of Weingarten subgraph $\mc H$, however, we shall see in Corollary \ref{H-indep} that this is not the case. 
\end{remark}

Let us now state the main theorem of the paper.

\begin{theorem} \label{thm:Laurent-series}
Fix once and for all a Weingarten subgraph $\mc H \subset \mc G$ and fix $p,q \in NC_2(2k)$  Then the Weingarten function $d \mapsto \Wg_d(p,q)$ admits the following absolutely convergent Laurent series expansion    
\begin{align}\label{laurent-expansion}
\Wg_{d}(p,q) = (-1)^{|p \vee q| + k}\sum_{r \ge 0} m_r(p,q)d^{-L(p,q) - 2r} \qquad \Big(|d| > 2\cos\big(\frac{\pi}{k+1}\big)\Big),
\end{align}
where $L(p,q)  \in \N_0$ is the distance from $(p,q)$ to $(\emptyset, \emptyset)$ in the Weingarten graph $\mc G$, and $m_r(p,q)$ is the number of paths of length $L(p,q)+2r$ in $\mc H$ as defined in Definition \ref{m_r}.

In particular, the leading order  term of $\Wg_{d}(p,q)$ is given by
\[
\Wg_d(p,q) \sim m_0(p,q)(-1)^{k + |p\vee q|}d^{-L(p,q)} \ne 0 \qquad (|d| \to \infty). 
\]
\end{theorem}

\begin{corollary} \label{H-indep} 
The number $m_r(p,q)$ of paths of length $L(p,q) +2r$ from a vertex $(p,q)$ to $(\emptyset,\emptyset)$ in any Weingarten subgraph $\mc H \subseteq \mc G$  is independent of the choice of $\mc H$.
\end{corollary}

\begin{proof}
This is an immediate consequence of the uniqueness of the coefficients of a Laurent series expansion for an analytic function on an annulus.
\end{proof}

\begin{example} \label{ex:workout}
As an illustration of Theorem \ref{thm:Laurent-series}, let us compute $\Wg_d(p,q)$, where $p = \{1,4\}\{2,3\}\{5,6\}$, $q = \{1,6\}\{2,5\}\{3,4\} \in NC_2(6)$.  Using downward (resp. upward) facing arches to depict the pairs of $p$ (resp $q$) in the interval $\{1,2, \ldots, 6\} = [6]$, we can graphically represent the pair $(p,q)$ with the following arch diagram.
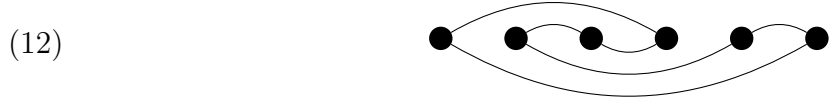
\begin{equation}
	\begin{tikzpicture}[baseline=(current  bounding  box.center),
			wh/.style={circle,draw=black,thick,inner sep=1mm},
			bl/.style={circle,draw=black,fill=black,thick,inner sep=1mm}]
		\node (1) at (0,0) [bl] {};
		\node (2) at (1,0) [bl] {};
		\node (3) at (2,0) [bl] {};
		\node (4) at (3,0) [bl] {};
		\node (5) at (4,0) [bl] {};
		\node (6) at (5,0) [bl] {};
		
\draw (1) to [bend right] (6);
\draw (2) to [bend right] (5);
\draw (3) to [bend right] (4);

\draw (1) to [bend left] (4);
\draw (2) to [bend left] (3);
\draw (5) to [bend left] (6);

	\end{tikzpicture}
\end{equation}
To compute $\Wg_d(p,q)$, we choose a Weingarten subraph $\mc H \subset \mc G$ and draw the component of $\mc H$ that is relevant to the pair $(p,q)$ we started with.  In this example, we make the following choice for this component of $\mc H$ (here the blue arrows indicate the directed edges that appear and the white nodes connected by a starred arch indicate the choice of interval for the Weingarten orthogonality relation taken in forming $\mc H$):

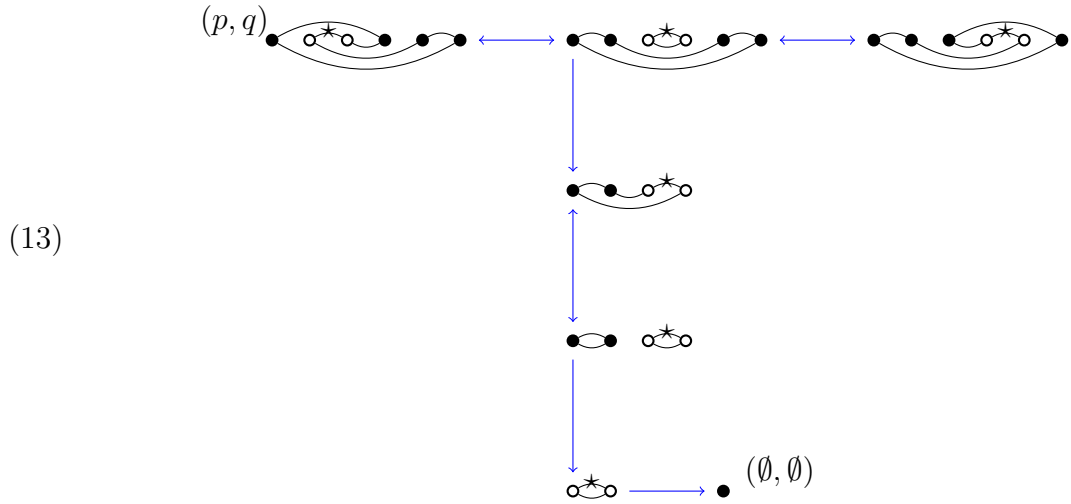
\begin{equation} \label{graph1}
	\begin{tikzpicture}[baseline=(current  bounding  box.center),
			wh/.style={circle,draw=black,thick,inner sep=.5mm},
			bl/.style={circle,draw=black,fill=black,thick,inner sep=.5mm}, scale = 0.5]
		\node (1) at (-8,0) [bl] {};
		\node (2) at (-7,0) [wh] {};
		\node (3) at (-6,0) [wh] {};
		\node (4) at (-5,0) [bl] {};
		\node (5) at (-4,0) [bl] {};
		\node (6) at (-3,0) [bl] {};
		
\draw (1) to [bend right] (6);
\draw (2) to [bend right] (5);
\draw (3) to [bend right] (4);

\draw (1) to [bend left] (4);
\draw (2) to [bend left] (3);
\draw (5) to [bend left] (6);

\node () at (-9,0.5) {$(p,q)$};		

\node (a) at (-6.5,0.25) {$\star$};

\draw [<->, color=blue]
		(-2.5,0) -- (-0.5,0);

		\node (11) at (0,0) [bl] {};
		\node (12) at (1,0) [bl] {};
		\node (13) at (2,0) [wh] {};
		\node (14) at (3,0) [wh] {};
		\node (15) at (4,0) [bl] {};
		\node (16) at (5,0) [bl] {};
		
\draw (11) to [bend right] (16);
\draw (12) to [bend right] (15);
\draw (13) to [bend right] (14);

\draw (11) to [bend left] (12);
\draw (13) to [bend left] (14);
\draw (15) to [bend left] (16);

\draw [<->, color=blue]
		(5.5,0) -- (7.5,0);

\node (b) at (2.5,0.25) {$\star$};		

		\node (21) at (8,0) [bl] {};
		\node (22) at (9,0) [bl] {};
		\node (23) at (10,0) [bl] {};
		\node (24) at (11,0) [wh] {};
		\node (25) at (12,0) [wh] {};
		\node (26) at (13,0) [bl] {};
		
\draw (21) to [bend right] (26);
\draw (22) to [bend right] (25);
\draw (23) to [bend right] (24);

\draw (21) to [bend left] (22);
\draw (23) to [bend left] (26);
\draw (24) to [bend left] (25);

\draw [->, color=blue]
		(0,-.5) -- (0,-3.5);

\node (c) at (11.5,0.25) {$\star$};	

	
		\node (31) at (0,-4) [bl] {};
		\node (32) at (1,-4) [bl] {};
		\node (33) at (2,-4) [wh] {};
		\node (34) at (3,-4) [wh] {};

\draw (31) to [bend right] (34);
\draw (32) to [bend right] (33);

\draw (31) to [bend left] (32);
\draw (33) to [bend left] (34);

\draw [<->, color=blue]
		(0,-4.5) -- (0,-7.5);

\node (d) at (2.5,-3.75) {$\star$};		


		\node (41) at (0,-8) [bl] {};
		\node (42) at (1,-8) [bl] {};
		\node (43) at (2,-8) [wh] {};
		\node (44) at (3,-8) [wh] {};

\draw (41) to [bend right] (42);
\draw (43) to [bend right] (44);

\draw (41) to [bend left] (42);
\draw (43) to [bend left] (44);

\draw [->, color=blue]
		(0,-8.5) -- (0,-11.5);

\node (e) at (2.5,-7.75) {$\star$};		


		\node (51) at (0,-12) [wh] {};
		\node (52) at (1,-12) [wh] {};

\draw (51) to [bend right] (52);

\draw (51) to [bend left] (52);

\draw [->, color=blue]
		(1.5,-12) -- (3.5,-12);

\node (d) at (0.5,-11.75) {$\star$};		

\node (121) at (4,-12) [bl] {};
		\node (121) at (5.5,-11.5) {$(\emptyset, \emptyset)$};

	\end{tikzpicture}
\end{equation}

From the graph \eqref{graph1}, we immediately see that $L(p,q) = 5$ and there is only one path from $(p,q)$ to $(\emptyset, \emptyset)$ with this length.  Moreover, any path of length $L(p,q) +2r$ from $(p,q)$ to $(\emptyset, \emptyset)$ corresponds to a choice of $0  \le s \le r$ loops to traverse within the $NC_2(6) \times NC_2(6)$-component of $\mc H$, followed by $r-s$ traversals of the single loop in the $NC_2(4) \times NC_2(4)$-component of $\mc H$.  Counting the number of such distinct choices easily gives 
\[
m_r(p,q) = \sum_{s=0}^r 2^s = 2^{r+1}-1 \qquad (r \ge 0),
\]
and  consequently, we get 
\[
\Wg_d(p,q) = (-1)^{|p \vee q| + k}\sum_{r \ge 0} m_r(p,q)d^{-L(p,q) - 2r} = \sum_{r \ge 0} (2^{r+1}-1)d^{-5 - 2r}
\]

\begin{remark}
At this point the reader may object to the fact that in Example \ref{ex:workout}, we did not explain why the graph \eqref{graph1} is indeed a (component of a) Weingarten subgraph $\mc H$.  In particular condition $(\mc H2)$ for $\mc H$ has not been explicitly verified.  In this particular example it is easy to verify this by hand.  Remarkably however, it turns out to follow from the proof of Theorem \ref{thm:Laurent-series} that condition $(\mc H2)$ {\it does not need to be verified}.  Indeed, in Section \ref{on-sg} we shall see that any subgraph $\mc H' \subset \mc G$ constructed according to the rules defining a Weingarten subgraph {\it without insisting on condition $(\mc H2)$} turns out to automatically satisfy condition $(\mc H2)$ anyway.  In particular, $\mc H'$ is automatically a Weingarten subgraph.  This shows that one has considerable ease and flexibility in constructing Weingarten subgraphs.
\end{remark}

\end{example}

Before giving the proof of Theorem \ref{thm:Laurent-series}, we first present our main application of this result, stating that the coefficients of any element of the dual basis associated to the Temperley-Lieb diagram basis $\TL_k(d)$ is non-zero.  

\begin{theorem} \label{thm:non-zero-coefficients}
Let $\hat D_p \in \TL_k(d)$ be the basis element dual to the Temperley-Lieb diagram $D_p$, with $p \in NC_2(2k)$ and loop parameter $d$. 
Then every coefficient of $\hat D_p$ in the $D_q$-basis expansion is generically non-zero, in the sense that $\Wg_d(p,q) = 0$ for at most finitely many $d \in \C$.  More precisely, we have $\Wg_d(p,q) \ne 0$ when
\[
d  \in \R \backslash \Big[-2\cos\big(\frac{\pi}{k+1}\big), 2\cos\big(\frac{\pi}{k+1}\big)\Big] \quad \text{or} \quad |d| \text{ is sufficiently large}.
\]
\end{theorem}

\begin{proof}
If $d  \in \R \backslash \Big[-2\cos\big(\frac{\pi}{k+1}\big), 2\cos\big(\frac{\pi}{k+1}\big)\Big]$, then Theorem \ref{thm:Laurent-series} implies that \[|\Wg_d(p,q)| = \sum_{r \ge 0}m_r(p,q)|d|^{-L(p,q)-2r} > 0  \quad \text{since $m_0(p,q) \ne 0$}.\] For the case of $|d| \to \infty$, Theorem \ref{thm:Laurent-series} gives the asymptotic \[\Wg_d(p,q)\sim m_0(p,q)(-1)^{k + |p\vee q|}d^{-L(p,q)} \ne 0,\] and we are done.    
\end{proof}

We now present the proof of Theorem \ref{thm:Laurent-series}.

\begin{proof}[Proof of Theorem \ref{thm:Laurent-series}]

To begin with, we note that for any  $k \in \N_0$ and any $p,q \in NC_2(2k)$, the function $d \mapsto \Wg_{d}(p,q)$ is a rational function in the variable $d \in \C$.  In fact, the determinant of the gram matrix $G_{d}$ is well-known (see for example in \cite[Theorem 1]{DiFr98}), and it follows from that result and the definition of the matrix inverse in terms of cofactors that the poles of $\Wg_{d}(p,q)$ always lie in the interval $\Big[-2\cos\big(\frac{\pi}{k+1}\big), 2\cos\big(\frac{\pi}{k+1}\big)\Big]$.  As a consequence, the rational function $d \mapsto \Wg_{d}(p,q)$ is analytic on the annulus $\mathbb A = \{d \in \C: |d| > 2\cos\big(\frac{\pi}{k+1}\big)\}$ and  has an absolutely convergent Laurent series expansion there.  

To determine this Laurent series, we recall from elementary complex variable theory that it suffices to determine the evaluation of the Laurent series for $\Wg_d(p,q)$ along a sequence of points in $\mathbb A$ tending to infinity.    For our purposes, it will be convenient to take the sequence of points $\{d \in \N: d \ge  k+1\}$, whenever $(p,q) \in NC_2(2k)^2$.

\begin{notation}
In order to simplify some notation, we will work with the following re-signed Weingarten functions
\[\widetilde{\Wg}_d(p,q) := (-1)^{k+ |p \vee q|}\Wg_d(p,q).\]
\end{notation}

Let us now fix once and for all a Weingarten subgraph $\mc H \subset \mc G$, $k \ge 1$, $p,q \in NC_2(2k)$ and $d \in \N$ with $d \ge k+1$.  Consider the distinguished Weingarten orthogonality relation \eqref{HE-meaning} associated to $(p,q)$ by $\mc H$.  Multiplying that equation by $(-1)^{k+|p \vee q|}d^{-1}$ and rearranging terms, we obtain the following equivalent equation (with the help of Proposition \ref{prop:parity}):
\begin{align} \label{two}
\widetilde{\Wg}_d(p,q)&= d^{-1} \sum_{((p,q),(p_1,q_1)) \in E_\mc H} \widetilde{\Wg}_d(p_1,q_1).
\end{align}

We note that the number of terms appearing in \eqref{two} is at most $k$.  (This follows from the fact any $z \in NC_2(2k)$ can have at most $k-1$ non-crossing neighbors associated to a fixed interval $\{t,t+1\} \in z$.  In particular, the degree any vertex $(p,q) \in V_\mc H$ is at most $k$).  Note also that for each $(p_1,q_1)$ appearing in \eqref{two}, we have $L(p_1,q_1) = L(p,q) \pm 1$ with at least one vertex $(p_1,q_1)$ satisfying $L(p_1,q_1) = L(p,q) - 1$.

We now repeatedly apply equation \eqref{two} to each term on the right hand side of \eqref{two}.  In particular, after $1 \le s < L(p,q)$ iterations, equation \eqref{two} gets transformed into the equation
\begin{align}\label{three}
\widetilde{\Wg}_d(p,q)&= d^{-s} \sum_{((p_{s-1},q_{s-1}),(p_s,q_s)) \in E_\mc H} \widetilde{\Wg}_d(p_s,q_s),
\end{align}
where the edges  $((p_{s-1},q_{s-1}),(p_s,q_s)) \in E_\mc H$ appearing above correspond to all the paths $(p,q)(p_1,q_1)\ldots (p_{s-1},q_{s-1})(p_s,q_s)$ of length $s$ in $\mc H$ with initial point $(p,q)$ that arise from the edge choices made in defining $\mc H$. Note that when $s=L(p,q)$, this will be the first time that we will produce a path $(p,q)(p_1,q_1)\ldots (p_{s-1},q_{s-1})(p_s,q_s)$ in $\mc H$ of length $L(p,q)$ with endpoint $(p_s,q_s) = (\emptyset,\emptyset)$.  Recalling that $m_0(p,q) \in \N$ denotes the number of such paths in $\mc H$, we can equivalently write \eqref{three} as 
\begin{align} \label{four}
\widetilde{\Wg}_d(p,q)&= m_0(p,q) d^{-L(p,q)} +  d^{-L(p,q)}\sum_{\substack{((p_{s-1},q_{s-1}),(p_s,q_s)) \in E_\mc H \\ (p_s,q_s) \ne (\emptyset, \emptyset)}} \widetilde{\Wg}_d(p_s,q_s),
\end{align}  
Continuing to apply \eqref{two} to the remaining $\widetilde{\Wg}_d$ terms on the right hand side, we inductively obtain the following general formula after  $s= L(p,q) + 2N$ iterations (with $N \in \N_0$).  
\begin{align} \label{five}
\widetilde{\Wg}_d(p,q)&=\sum_{r=0}^N m_r(p,q) d^{-L(p,q)-2r}+  d^{-L(p,q)-2N}\sum_{\substack{((p_{s-1},q_{s-1}),(p_s,q_s)) \in E_\mc H \\ (p_s,q_s) \ne (\emptyset, \emptyset)}} \widetilde{\Wg}_d(p_{s},q_{s}),
\end{align}  
where $m_r(p,q) \in \N_0$  is the number of directed paths of  length $L(p,q) +2r$ from $(p,q)$ to $(\emptyset, \emptyset)$ contained in $\mc H$. 
Note in particular that no paths of length $L(p,q) +2r+1$ exist from $(p,q)$ to $(\emptyset,\emptyset)$, thanks to Corollary \ref{length-parity}.

From equation \eqref{five}, it follows that we obtain the desired formula 
\[
\widetilde{\Wg}_d(p,q)=\sum_{r=0}^\infty m_r(p,q) d^{-L(p,q)-2r} \qquad (d \ge k+1),
\] provided we can show that 
\[
\lim_{N\to \infty} d^{-L(p,q)-2N}\Big|\sum_{\substack{((p_{s-1},q_{s-1}),(p_s,q_s)) \in E_\mc H \\ (p_s,q_s) \ne (\emptyset, \emptyset)}} \widetilde{\Wg}_d(p_{s},q_{s})\Big| = 0.
\]
To this end, note that after $s = L(p,q)+ 2N$ iterations there are at most $k^s$ terms being summed in the above remainder term.  Moreover, the set $\{\widetilde{\Wg}_d(p_{s},q_{s})\}$  of numbers being summed is uniformly bounded in $s$, since the $q_s, p_s$ that give rise to these terms are constrained to  live in $\bigcup_{0 \le j \le k} NC_2(2j)$.  Thus the H\"older inequality gives
\[
d^{-L(p,q)-2N}\Big|\sum_{\substack{((p_{s-1},q_{s-1}),(p_s,q_s)) \in E_\mc H \\ (p_s,q_s) \ne (\emptyset, \emptyset)}} \widetilde{\Wg}_d(p_{s},q_{s})\Big| \le \Big(\frac{k}{d}\Big)^{L(p,q)+2N} \sup_{(p_s,q_s) \ne (\emptyset, \emptyset)}|\widetilde{\Wg}_d(p_{s},q_{s})| \to 0,
\]  completing the proof.
\end{proof}

\begin{remark}
The idea of counting paths as it is done in the proof of Theorem \ref{thm:Laurent-series} is also heavily used in subfactor theory, for example for the computation of dimensions of relative commutants.  See for example \cite{GoDeJo89, JoSu97}.  We are not able at this point to interpret the numbers $m_r(p,q)$ that we introduce as the dimension of an object of something alike, but it is natural to speculate that there is one such interpretation.
 
Our graphs are different from principal graphs of subfactors because they are oriented, but we believe that there is a relation that deserves further investigation. The fact that the choice of the Weingarten subgraph does not affect the computation of the Weingarten function possibly hints at the fact that there is a ``type'' of graph associated to a pair of non-crossing partitions, possibly related to known series of principal graphs. Note also that just as for the graphs arising in subfactor theory, there seems to be some duality present in our graphs. 
\end{remark}

\subsection{On the problem of selecting a Weingarten subgraph} \label{on-sg}
A natural question that arises from the above analysis is: {\it In practice, how does one efficiently select a Weingarten subgraph $\mc H$?}  In particular, condition $(\mc H2)$ in the definition of a Weingarten subgraph $\mc H \subset \mc G$ seems to require ``global information'' about the graph $\mc G$ in order to select orthogonality relations that produce edges that decrease the distance from a vertex $(p,q)$ to $(\emptyset,\emptyset)$.  The following proposition asserts the highly non-obvious fact that condition $(\mc H2)$ essentially ``comes for free'' in the construction of a Weingarten subgraph $\mc H$.   This has profound applications for the practical implementation of Theorem \ref{thm:Laurent-series}, since it implies that $\mc H$ can be constructed from purely local information about vertices.

\begin{proposition}\label{noH2}
Let $(p,q) \ne (\emptyset, \emptyset)$ be a vertex in the Weingarten graph $\mc G$.  For any choice of Weingarten orthogonality relation of the form \eqref{Wg-relation1}-\eqref{Wg-relation2} at $(p,q)$, the resulting collection of edges $\{((p,q), (p_1,q_1))\} \subset E_\mc G$ associated to this relation contains at least one element $((p,q), (p_1,q_1))$ satisfying 
\[L(p_1,q_1)  = L(p,q) -1. \]
In particular, in the process of selecting a Weingarten graph $\mc H \subset \mc G$, condition $(\mc H2)$ is automatically satisfied for any choice of Weingarten orthogonality relation. 
\end{proposition}

\begin{proof}
Suppose, to get a contradiction, that there exists a $(p,q) \ne (\emptyset, \emptyset)$ and a Weingarten orthogonaliy relation at $(p,q)$ such that $L(p_1,q_1)  = L(p,q) +1$ for all the resulting edges $((p,q), (p_1,q_1))$ associated to this relation.  Consider the (positive) quantity $\widetilde{\Wg}_d(p,q)$ defined in the proof of Theorem \ref{thm:Laurent-series}.  On the one hand, Theorem \ref{thm:Laurent-series} gives 
\begin{align} \label{aa} \widetilde{\Wg}_d(p,q) = \sum_{r \ge 0} m_r(p,q)d^{-L(p,q)-2r} = m_0(p,q) d^{-L(p,q)} + O(d^{-L(p,q)-2}).
\end{align}
On the other hand, repeating the argument in the proof of that theorem, we can equally write 
\[\widetilde{\Wg}_d(p,q) = \sum_{} \widetilde{\Wg}_d(p_1,q_1),\]
where the above sum runs over all edges $((p,q), (p_1,q_1)) \in E_\mc G$ that are associated to our chosen orthogonality relation.  Applying Theorem \ref{thm:Laurent-series} again to the terms on the right side of the above, we get 
\begin{align}\label{bb}
\widetilde{\Wg}_d(p,q) &= \sum_{r \ge 0} \sum_{(p_1,q_1)}  m_r(p_1,q_1)d^{-L(p_1,q_1)-2r} =  \Big(\sum_{(p_1,q_1)}  m_r(p_1,q_1)\Big)d^{-L(p,q)-1} + O(d^{-L(p,q)-3}). 
\end{align}     
Clearly the asymptotics \eqref{aa} and \eqref{bb} contradict each other, completing the proof.
\end{proof}

We now conclude the paper by discussing applications of our results to the asymptotic theory of the free orthogonal Weingarten functions.

\subsection{Optimal Weingarten estimates}
As mentioned in the introduction, the problem of computing the large $d$ asymptotics of the Weingarten function has recieved a lot of attention in recent years in the context of distributional symmetries in free probability.  Of particular interest there is the problem of identifying the leading non-zero coefficient in the Laurent expansion of $\Wg_d(p,q)$  centered at the origin (provided such a coefficient exists).  Partial results along these have been obtained previously, and are usually given the term {\it Weingarten estimates}.  See for example \cite{BaCo07, BaCuSp12, CuSp11}. It is important to note, however, that Theorem \ref{thm:Laurent-series} provides the first  complete answer to this problem by showing that such a non-zero coefficient must always exist and by describing the degree of the leading term in $\Wg_d(p,q)$ by way of graph theoretical data.   More precisely, the results of \cite{BaCuSp12} give Weingarten estimates of the form \[\Wg_{d}(p,q) =\Bigg\{\begin{matrix}O(d^{-2k + |p \vee q|}), & p \ne q \\
d^{-k} + O(d^{-k-2}), & p=q\end{matrix} \qquad (d \to \infty),\]  while  in Theorem 4.6 of \cite{CuSp11}, a sharper result is obtained which obtains the leading non-zero coefficient in the Laurent series expansion of $\Wg_{d}(p,q)$ precisely when a certain Moebius function they introduce on $NC_2(2k) \times NC_2(2k)$ is non-zero at $(p,q)$.  Unfortunately these prior results are far from covering all values.  Let us conclude by presenting a simple example to illustrate this.

\begin{example} \label{ex:CS-fail}
Let $k=4$, $p = \{1,6\}\{2,5\}\{3,4\}\{7,8\}$, and $q = \{1,2\}\{3,8\}\{4,7\}\{5,6\}$.  Using  arch diagrams as in Example \ref{ex:workout}, we can depict the pair $(p,q)$ as follows: 
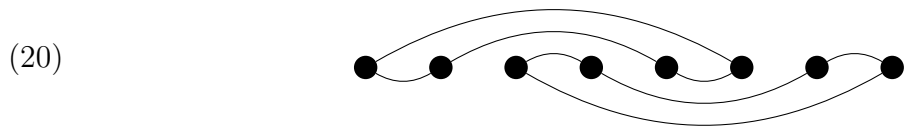
\begin{equation}
	\begin{tikzpicture}[baseline=(current  bounding  box.center),
			wh/.style={circle,draw=black,thick,inner sep=1mm},
			bl/.style={circle,draw=black,fill=black,thick,inner sep=1mm}]
		\node (1) at (0,0) [bl] {};
		\node (2) at (1,0) [bl] {};
		\node (3) at (2,0) [bl] {};
		\node (4) at (3,0) [bl] {};
		\node (5) at (4,0) [bl] {};
		\node (6) at (5,0) [bl] {};
		\node (7) at (6,0) [bl] {};
		\node (8) at (7,0) [bl] {};
\draw (1) to [bend right] (2);
\draw (3) to [bend right] (8);
\draw (4) to [bend right] (7);
\draw (5) to [bend right] (6);

\draw (1) to [bend left] (6);
\draw (2) to [bend left] (5);
\draw (3) to [bend left] (4);
\draw (7) to [bend left] (8);
	
	\end{tikzpicture}
\end{equation}
Evidently $|p\vee q| = 2$, and it easily follows that the results of \cite{BaCuSp12, CuSp11} yield a Weingarten asymptotic of the form \[\Wg_{d}(p,q) = O(d^{-2k +|p \vee q|}) =  O(d^{-6}).\]  But in fact the leading order of $\Wg_d(p,q)$ turns out to be much smaller in this example.  Using the notation of Theorem \ref{thm:Laurent-series}, one actually has $m_0(p,q) = 1$, $L(p,q) = 8$, and therefore 
\[
\Wg_{d}(p,q) = d^{-8} + O(d^{-10}).
\]
To see why this is the case, we proceed as in Example \ref{ex:workout} by  choosing a Weingarten subraph $\mc H$ and drawing the component of $\mc H$ relevant to $(p,q)$.  The following picture depicts such a component (where as before the blue arrows indicate the directed edges that appear and the white nodes connected by a starred arch indicate the choice of interval for the Weingarten orthogonality relation in forming $\mc H$).

\begin{equation}\label{graph2}
	\begin{tikzpicture}[baseline=(current  bounding  box.center),
			wh/.style={circle,draw=black,thick,inner sep=.5mm},
			bl/.style={circle,draw=black,fill=black,thick,inner sep=.5mm}, scale = 0.35]
		\node (11) at (0,0) [bl] {};
		\node (12) at (1,0) [bl] {};
		\node (13) at (2,0) [wh] {};
		\node (14) at (3,0) [wh] {};
		\node (15) at (4,0) [bl] {};
		\node (16) at (5,0) [bl] {};
		\node (17) at (6,0) [bl] {};
		\node (18) at (7,0) [bl] {};
\draw (11) to [bend right] (12);
\draw (13) to [bend right] (18);
\draw (14) to [bend right] (17);
\draw (15) to [bend right] (16);
\draw (11) to [bend left] (16);
\draw (12) to [bend left] (15);
\draw (13) to [bend left] (14);
\draw (17) to [bend left] (18);

\node () at (2.5,0.25) {$\star$};		
\node () at (-1.5,0.5) {$(p,q)$};		

\draw [<->, color=blue]
		(7.5,0) -- (9.5,0);

\draw [<->, color=blue]
		(17.5,0) -- (19.5,0); 

		\node (21) at (10,0) [bl] {};
		\node (22) at (11,0) [wh] {};
		\node (23) at (12,0) [wh] {};
		\node (24) at (13,0) [bl] {};
		\node (25) at (14,0) [bl] {};
		\node (26) at (15,0) [bl] {};
		\node (27) at (16,0) [bl] {};
		\node (28) at (17,0) [bl] {};
\draw (21) to [bend right] (22);
\draw (23) to [bend right] (28);
\draw (24) to [bend right] (27);
\draw (25) to [bend right] (26);

\draw (21) to [bend left] (26);
\draw (22) to [bend left] (23);
\draw (24) to [bend left] (25);
\draw (27) to [bend left] (28);

\node () at (11.5,0.25) {$\star$};		

		\node (31) at (20,0) [wh] {};
		\node (32) at (21,0) [wh] {};
		\node (33) at (22,0) [bl] {};
		\node (34) at (23,0) [bl] {};
		\node (35) at (24,0) [bl] {};
		\node (36) at (25,0) [bl] {};
		\node (37) at (26,0) [bl] {};
		\node (38) at (27,0) [bl] {};
\draw (31) to [bend right] (32);
\draw (33) to [bend right] (38);
\draw (34) to [bend right] (37);
\draw (35) to [bend right] (36);

\draw (31) to [bend left] (32);
\draw (33) to [bend left] (36);
\draw (34) to [bend left] (35);
\draw (37) to [bend left] (38);

\node () at (20.5,0.25) {$\star$};		

		\node (3'1) at (30,0) [bl] {};
		\node (3'2) at (31,0) [bl] {};
		\node (3'3) at (32,0) [bl] {};
		\node (3'4) at (33,0) [wh] {};
		\node (3'5) at (34,0) [wh] {};
		\node (3'6) at (35,0) [bl] {};
		\node (3'7) at (36,0) [bl] {};
		\node (3'8) at (37,0) [bl] {};
\draw (3'1) to [bend right] (3'2);
\draw (3'3) to [bend right] (3'8);
\draw (3'4) to [bend right] (3'7);
\draw (3'5) to [bend right] (3'6);

\draw (3'1) to [bend left] (3'8);
\draw (3'2) to [bend left] (3'7);
\draw (3'3) to [bend left] (3'6);
\draw (3'4) to [bend left] (3'5);

\node () at (33.5,0.25) {$\star$};		

\draw [->, color=blue]
		(27.5,0) -- (29.5,0); 

\draw [->, color=blue]
		(20,-.5) -- (20,-3.5); 

\draw [->, color=blue]
		(37.5,0) -- (40.5,0); 

\node (121) at (42,0) {$\cdots$ };


		\node (41) at (20,-4) [bl] {};
		\node (42) at (21,-4) [wh] {};
		\node (43) at (22,-4) [wh] {};
		\node (44) at (23,-4) [bl] {};
		\node (45) at (24,-4) [bl] {};
		\node (46) at (25,-4) [bl] {};
\draw (41) to [bend left] (44);
\draw (42) to [bend left] (43);
\draw (45) to [bend left] (46);

\draw (41) to [bend right] (46);
\draw (42) to [bend right] (45);
\draw (43) to [bend right] (44);

\node () at (21.5,-3.75) {$\star$};

\draw [<->, color=blue]
		(25.5,-4) -- (27.5,-4);

		\node (51) at (28,-4) [bl] {};
		\node (52) at (29,-4) [bl] {};
		\node (53) at (30,-4) [wh] {};
		\node (54) at (31,-4) [wh] {};
		\node (55) at (32,-4) [bl] {};
		\node (56) at (33,-4) [bl] {};
		
\draw (51) to [bend right] (56);
\draw (52) to [bend right] (55);
\draw (53) to [bend right] (54);

\draw (51) to [bend left] (52);
\draw (53) to [bend left] (54);
\draw (55) to [bend left] (56);

\node () at (30.5,-3.75) {$\star$};	

\draw [<->, color=blue]
		(33.5,-4) -- (35.5,-4);

		\node (61) at (36,-4) [bl] {};
		\node (62) at (37,-4) [bl] {};
		\node (63) at (38,-4) [bl] {};
		\node (64) at (39,-4) [wh] {};
		\node (65) at (40, -4) [wh] {};
		\node (66) at (41,-4) [bl] {};
		
\draw (61) to [bend right] (66);
\draw (62) to [bend right] (65);
\draw (63) to [bend right] (64);

\draw (61) to [bend left] (62);
\draw (63) to [bend left] (66);
\draw (64) to [bend left] (65);

\node () at (39.5,-3.75) {$\star$};	

\draw [->, color=blue]
		(28,-4.5) -- (28,-7.5);

	
		\node (71) at (28,-8) [bl] {};
		\node (72) at (29,-8) [bl] {};
		\node (73) at (30,-8) [wh] {};
		\node (74) at (31,-8) [wh] {};

\draw (71) to [bend right] (74);
\draw (72) to [bend right] (73);

\draw (71) to [bend left] (72);
\draw (73) to [bend left] (74);

\node () at (30.5,-7.75) {$\star$};	

\draw [<->, color=blue]
		(28,-8.5) -- (28,-11.5);


		\node (81) at (28,-12) [bl] {};
		\node (82) at (29,-12) [bl] {};
		\node (83) at (30,-12) [wh] {};
		\node (84) at (31,-12) [wh] {};

\draw (81) to [bend right] (82);
\draw (83) to [bend right] (84);

\draw (81) to [bend left] (82);
\draw (83) to [bend left] (84);

\node () at (30.5,-11.75) {$\star$};	

\draw [->, color=blue]
		(28,-12.5) -- (28,-15.5);


		\node (91) at (28,-16) [wh] {};
		\node (92) at (29,-16) [wh] {};

\draw (91) to [bend right] (92);

\draw (91) to [bend left] (92);

\node () at (28.5,-15.75) {$\star$};	

\draw [->, color=blue]
		(29.5,-16) -- (31.5,-16);

\node (121) at (32,-16) [bl] {};
		\node (121) at (33.5,-15.5) {$(\emptyset, \emptyset)$};

	\end{tikzpicture}
\end{equation}
In the above picture, the dots on the top right corner indicate further vertices of $\mc H$ that lie in $NC_2(8) \times NC_2(8)$ that we have ommitted.  We are not concerned with these other vertices because they are not relevant for the computation of the leading order data $L(p,q)$ or $m_0(p,q)$ (they give rise to paths strictly longer than the single shortest path from $(p,q)$ to $(\emptyset, \emptyset)$ of length 8).  We thus conclude from inspection of this graph that $m_0(p,q) = 1$ and $L(p,q) = 8$.

\end{example}

\bibliographystyle{alpha}
\bibliography{brannan-biblio}

\def\cprime{$'$}
\begin{thebibliography}{GdlHJ89}

\bibitem[Abr08]{Ab08}
Samson Abramsky.
\newblock Temperley-{L}ieb algebra: from knot theory to logic and computation
  via quantum mechanics.
\newblock In {\em Mathematics of quantum computation and quantum technology},
  Chapman \& Hall/CRC Appl. Math. Nonlinear Sci. Ser., pages 515--558. Chapman
  \& Hall/CRC, Boca Raton, FL, 2008.

\bibitem[Ban96]{Ba96}
Teodor Banica.
\newblock Th\'eorie des repr\'esentations du groupe quantique compact libre
  {${\rm O}(n)$}.
\newblock {\em C. R. Acad. Sci. Paris S\'er. I Math.}, 322(3):241--244, 1996.

\bibitem[Ban97]{Ba97}
Teodor Banica.
\newblock Le groupe quantique compact libre {${\rm U}(n)$}.
\newblock {\em Comm. Math. Phys.}, 190(1):143--172, 1997.

\bibitem[BC07]{BaCo07}
Teodor Banica and Beno{\^{\i}}t Collins.
\newblock Integration over compact quantum groups.
\newblock {\em Publ. Res. Inst. Math. Sci.}, 43(2):277--302, 2007.

\bibitem[BC10]{BaCu10}
Teodor Banica and Stephen Curran.
\newblock Decomposition results for {G}ram matrix determinants.
\newblock {\em J. Math. Phys.}, 51(11):113503, 14, 2010.

\bibitem[BCS12]{BaCuSp12}
Teodor Banica, Stephen Curran, and Roland Speicher.
\newblock De {F}inetti theorems for easy quantum groups.
\newblock {\em Ann. Probab.}, 40(1):401--435, 2012.

\bibitem[BCZJ09]{BaCoZi09}
Teodor Banica, Benoit Collins, and Paul Zinn-Justin.
\newblock Spectral analysis of the free orthogonal matrix.
\newblock {\em Int. Math. Res. Not. IMRN}, (17):3286--3309, 2009.

\bibitem[BK16]{BrKi16}
Michael Brannan and Kay Kirkpatrick.
\newblock Quantum groups and generalized circular elements.
\newblock {\em Pacific J. Math.}, 282(1):35--61, 2016.

\bibitem[BS09]{BaSp}
Teodor Banica and Roland Speicher.
\newblock Liberation of orthogonal {L}ie groups.
\newblock {\em Adv. Math.}, 222(4):1461--1501, 2009.

\bibitem[Cai11]{Ca11}
Xuanting Cai.
\newblock A {G}ram determinant of {L}ickorish's bilinear form.
\newblock {\em Math. Proc. Cambridge Philos. Soc.}, 151(1):83--94, 2011.

\bibitem[CFS95]{CaFlSa95}
J.~Scott Carter, Daniel~E. Flath, and Masahico Saito.
\newblock {\em The classical and quantum 6{$j$}-symbols}, volume~43 of {\em
  Mathematical Notes}.
\newblock Princeton University Press, Princeton, NJ, 1995.

\bibitem[CS11]{CuSp11}
Stephen Curran and Roland Speicher.
\newblock Asymptotic infinitesimal freeness with amalgamation for {H}aar
  quantum unitary random matrices.
\newblock {\em Comm. Math. Phys.}, 301(3):627--659, 2011.

\bibitem[Cur10]{Cu}
Stephen Curran.
\newblock Quantum rotatability.
\newblock {\em Trans. Amer. Math. Soc.}, 362(9):4831--4851, 2010.

\bibitem[DF98]{DiFr98}
P.~Di~Francesco.
\newblock Meander determinants.
\newblock {\em Comm. Math. Phys.}, 191(3):543--583, 1998.

\bibitem[DRW16]{DeRoWa16}
Colleen Delaney, Eric~C. Rowell, and Zhenghan Wang.
\newblock {L}ocal unitary representations of the braid group and their
  applications to quantum computing.
\newblock Preprint, arXiv:1604.06429, 2016.

\bibitem[FK97]{FrKh97}
Igor~B. Frenkel and Mikhail~G. Khovanov.
\newblock Canonical bases in tensor products and graphical calculus for {$U\sb
  q({ s}{ l}\sb 2)$}.
\newblock {\em Duke Math. J.}, 87(3):409--480, 1997.

\bibitem[GdlHJ89]{GoDeJo89}
Frederick~M. Goodman, Pierre de~la Harpe, and Vaughan F.~R. Jones.
\newblock {\em Coxeter graphs and towers of algebras}, volume~14 of {\em
  Mathematical Sciences Research Institute Publications}.
\newblock Springer-Verlag, New York, 1989.

\bibitem[Jon83]{Jo83}
V.~F.~R. Jones.
\newblock Index for subfactors.
\newblock {\em Invent. Math.}, 72(1):1--25, 1983.

\bibitem[JS97]{JoSu97}
V.~Jones and V.~S. Sunder.
\newblock {\em Introduction to subfactors}, volume 234 of {\em London
  Mathematical Society Lecture Note Series}.
\newblock Cambridge University Press, Cambridge, 1997.

\bibitem[Kau87]{Ka87}
Louis~H. Kauffman.
\newblock State models and the {J}ones polynomial.
\newblock {\em Topology}, 26(3):395--407, 1987.

\bibitem[KL94]{KaLi94}
Louis~H. Kauffman and S{\'o}stenes~L. Lins.
\newblock {\em Temperley-{L}ieb recoupling theory and invariants of
  {$3$}-manifolds}, volume 134 of {\em Annals of Mathematics Studies}.
\newblock Princeton University Press, Princeton, NJ, 1994.

\bibitem[KS91]{KoSm91}
Ki~Hyoung Ko and Lawrence Smolinsky.
\newblock A combinatorial matrix in {$3$}-manifold theory.
\newblock {\em Pacific J. Math.}, 149(2):319--336, 1991.

\bibitem[Lic91]{Li91}
W.~B.~R. Lickorish.
\newblock Invariants for {$3$}-manifolds from the combinatorics of the {J}ones
  polynomial.
\newblock {\em Pacific J. Math.}, 149(2):337--347, 1991.

\bibitem[Mor15]{Mo15}
Scott Morrison.
\newblock A formula for the jones-wenzl projections.
\newblock Preprint, arXiv:1503.00384, 2015.

\bibitem[Ocn02]{Oc02}
Adrian Ocneanu.
\newblock The classification of subgroups of quantum {${\rm SU}(N)$}.
\newblock In {\em Quantum symmetries in theoretical physics and mathematics
  ({B}ariloche, 2000)}, volume 294 of {\em Contemp. Math.}, pages 133--159.
  Amer. Math. Soc., Providence, RI, 2002.

\bibitem[Rez02]{Re02}
Sarah~Anne Reznikoff.
\newblock {\em Representations of the {T}emperley-{L}ieb planar algebra}.
\newblock ProQuest LLC, Ann Arbor, MI, 2002.
\newblock Thesis (Ph.D.)--University of California, Berkeley.

\bibitem[Rez07]{Re07}
Sarah~A. Reznikoff.
\newblock Coefficients of the one- and two-gap boxes in the {J}ones-{W}enzl
  idempotent.
\newblock {\em Indiana Univ. Math. J.}, 56(6):3129--3150, 2007.

\bibitem[TL71]{TeLi71}
H.~N.~V. Temperley and E.~H. Lieb.
\newblock Relations between the ``percolation'' and ``colouring'' problem and
  other graph-theoretical problems associated with regular planar lattices:
  some exact results for the ``percolation'' problem.
\newblock {\em Proc. Roy. Soc. London Ser. A}, 322(1549):251--280, 1971.

\bibitem[VDW96]{VaWa96}
Alfons Van~Daele and Shuzhou Wang.
\newblock Universal quantum groups.
\newblock {\em Internat. J. Math.}, 7(2):255--263, 1996.

\bibitem[Wei78]{We78}
Don Weingarten.
\newblock Asymptotic behavior of group integrals in the limit of infinite rank.
\newblock {\em J. Mathematical Phys.}, 19(5):999--1001, 1978.

\bibitem[Wen87]{We87}
Hans Wenzl.
\newblock On sequences of projections.
\newblock {\em C. R. Math. Rep. Acad. Sci. Canada}, 9(1):5--9, 1987.

\bibitem[Wor87a]{Wo87b}
Stanis\l aw~L. Woronowicz.
\newblock Compact matrix pseudogroups.
\newblock {\em Comm. Math. Phys.}, 111(4):613--665, 1987.

\bibitem[Wor87b]{Wo87a}
Stanis\l aw~L. Woronowicz.
\newblock Twisted {${\rm SU}(2)$} group. {A}n example of a noncommutative
  differential calculus.
\newblock {\em Publ. Res. Inst. Math. Sci.}, 23(1):117--181, 1987.

\bibitem[Wor98]{Wo98}
S.~L. Woronowicz.
\newblock Compact quantum groups.
\newblock In {\em Sym\'etries quantiques ({L}es {H}ouches, 1995)}, pages
  845--884. North-Holland, Amsterdam, 1998.

\bibitem[Zha09]{Zh09}
Yong Zhang.
\newblock Braid group, {T}emperley-{L}ieb algebra, and quantum information and
  computation.
\newblock In {\em Advances in quantum computation}, volume 482 of {\em Contemp.
  Math.}, pages 49--89. Amer. Math. Soc., Providence, RI, 2009.

\end{thebibliography}

\end{document}